\documentclass[11pt]{article}

\usepackage{bm}
\usepackage[top=1in,bottom=1in,left=1in,right=1in]{geometry}

\usepackage{textgreek,mathtools}
\usepackage{amsmath}
\usepackage{graphicx,subfig}
\usepackage[colorinlistoftodos]{todonotes}
\usepackage[colorlinks=true, allcolors=blue]{hyperref}
\usepackage{amsmath}
\usepackage{upgreek}
\usepackage{amssymb}
\usepackage{amsfonts}
\usepackage{amsthm}
\usepackage{mathrsfs}
\usepackage{fontenc}
\usepackage{empheq}
\usepackage{enumerate}
\usepackage{comment}
\usepackage{appendix}
\usepackage{bbm}
\usepackage{algorithm,algorithmic}
\mathtoolsset{showonlyrefs}
\usepackage{enumerate}
\usepackage{makecell}
\usepackage{booktabs}
\usepackage{multirow}
\usepackage{siunitx}

\theoremstyle{plain}
\newtheorem{thm}{Theorem}[section]

\newtheorem{lem}[thm]{Lemma}
\newtheorem{prop}[thm]{Proposition}
\newtheorem{cor}[thm]{Corollary}
\newtheorem{assu}[thm]{Assumption}

\theoremstyle{definition}
\newtheorem{defn}[thm]{Definition}

\newtheorem{nota}[thm]{Notation}

\theoremstyle{remark}
\newtheorem{rem}[thm]{Remark}

\numberwithin{equation}{section}

\newcommand{\RR}{\mathbb{R}}

\newcommand{\PP}{\mathbb{P}}

\newcommand{\ud}{\mathrm{d}}

\newcommand{\half}{\frac{1}{2}}

\newcommand{\EPS}{\varepsilon}
\newcommand{\R}{\mathbb{R}}
\newcommand{\F}{\mathscr{F}}
\newcommand{\SET}[1]{\left\{#1\right\}}
\newcommand{\E}{\mathbb{E}}
\newcommand{\NORM}[1]{\Vert#1\Vert}
\newcommand{\A}{\mathscr{A}}

\newcommand{\transpose}{^{\operatorname{T}}}

\newcommand{\Tr}{\text{Tr}}
\newcommand{\N}{\mathbb{N}}
\newcommand{\X}{\mathscr{X}}

\title{Finite-Agent Stochastic Differential Games on Large Graphs: I. The Linear-Quadratic Case}

\author{Ruimeng Hu\thanks{Department of Mathematics, and Department of Statistics and Applied Probability, University of California, Santa Barbara, CA 93106-3080, {\em rhu@ucsb.edu}.} \and  Jihao Long\thanks{Institute for Advanced Algorithms Research, Shanghai, China, \em{longjh1998@gmail.com}.} \and Haosheng Zhou\thanks{Department of Statistics and Applied Probability, University of California, Santa Barbara, CA 93106-3110, \em{hzhou593@ucsb.edu}.}}
\date{\today}

\begin{document}

\maketitle

\begin{abstract}
In this paper, we study finite-agent linear-quadratic games on graphs. Specifically, we propose a comprehensive framework that extends the existing literature by incorporating heterogeneous and interpretable player interactions. Compared to previous works, our model offers a more realistic depiction of strategic decision-making processes. For general graphs, we establish the convergence of fictitious play, a widely-used iterative solution method for determining the Nash equilibrium of our proposed game model. Notably, under appropriate conditions, this convergence holds true irrespective of the number of players involved. For vertex-transitive graphs, we develop a semi-explicit characterization of the Nash equilibrium. Through rigorous analysis, we demonstrate the well-posedness of this characterization under certain conditions. We present numerical experiments that validate our theoretical results and provide insights into the intricate relationship between various game dynamics and the underlying graph structure.
\end{abstract}

\textbf{Key words:}
\textit{Stochastic differential game on graphs, Nash equilibrium, Convergence, Fictitious play, Vertex-transitive graph}

\maketitle


\section{Introduction}

Game theory, established in the 1940s \cite{von2007theory}, has become a premier method for modeling human interactions. Similarly, graph theory, which addresses pairwise relationships between objects, has gained significant interest since Euler's solution to the ``seven bridges of Königsberg'' problem. Graphs, representing various types of information, have been extensively used in modeling biological, informational, and social systems \cite{mason2007graph,riaz2011applications,majeed2020graph}. Thus, exploring the interdisciplinary synthesis of game theory and graph theory is crucial. Existing literature shows the efficacy of combining games and graphs in modeling significant phenomena, such as network security \cite{mavronicolas2005graph}, vehicle-to-vehicle communication \cite{hassija2020dagiov}, and advertising dynamics \cite{hafezalkotob2018cooperation}. We focus on stochastic differential games with graph structures, governed by systems of stochastic differential equations \cite{friedman1972stochastic,carmona2016lectures}, which model continuous-time games with continuous state and action spaces. Incorporating graph structures captures direct and indirect player interactions, introducing heterogeneity. These games have garnered substantial attention for understanding behaviors within economic and financial frameworks, including interbank lending \cite{capponi2020dynamic}, systemic risk \cite{carmona2013mean}, default contagion \cite{feinstein2019dynamic}, and credit networks \cite{nadtochiy2020mean}.

In the literature, there are two major types of stochastic differential games: cooperative games, where players work towards a collective goal, and non-cooperative games, where players aim to maximize their own advantages. In the context of non-cooperative game theory, which we focus on in this paper, one key solution concept is the Nash equilibrium (NE) \cite{nash1950equilibrium,nash1950non}, where no player can benefit by unilaterally changing their own strategy. Due to its significance, the analytic and numerical solutions of the NE have been core topics in non-cooperative game theory.

Regarding the analytic solution for the NE, linear-quadratic games represent the majority of game models that have a closed-form solution for the NE. The linear-quadratic flocking model \cite{carmona2018probabilistic}, the systemic risk model \cite{carmona2013mean}, and other models presented in \cite{bensoussan2016linear,lacker2022case,gao2023lqg} are notable examples. Despite their simple structures, linear-quadratic games are powerful enough to explain real-world phenomena and have been adopted in various fields \cite{yong2002leader,bensoussan2016linear,zhang2020stochastic,druagan2023game}. For a comprehensive discussion on linear-quadratic games, we refer readers to \cite{sun2020stochastic}. For a family of important examples of linear-quadratic games with mean-field interactions, see \cite{carmona2018probabilistic}.

For a linear-quadratic game, numerically solving for the NE can always be transformed into numerically solving the associated coupled Riccati equations. This task is relatively straightforward due to the existence of numerous efficient numerical schemes for solving ODEs. However, solving for the NE of a general stochastic differential game is challenging, especially when the game is fully nonlinear (with controlled diffusion terms), involves a relatively large number of players, and/or has high-dimensional states. This complexity gives rise to fictitious play (FP) as an efficient method for numerically solving for the NE of a non-cooperative game in an iterative way, initially introduced by Brown \cite{brown1949some,brown1951iterative}. The main idea behind FP is to recast the problem into a sequence of stochastic control problems faced by each player to be solved repeatedly. However, the convergence of FP to the NE is not guaranteed, even for tabular games \cite{shapley1963some}, which motivates research on the convergence of FP. In the context of continuous-time differential games, affirmative results include \cite{cardaliaguet2017learning,briani2018stable} for mean-field games and \cite{hu2019deep,han2022convergence} for finite-player games.

A closely related research direction is graphon mean-field games (graphon MFGs), which extend classical MFGs by incorporating heterogeneous interactions through graphons---limiting objects representing the structure of dense networks \cite{bayraktar2023propagation,caines2019graphon, caines2021graphon, gao2023lqg}. Classical MFGs \cite{lasry2007mean} approximate finite-agent games in the limit of an infinite number of homogeneous players and have been widely used to model human behavior in various fields, including Aiyagari’s growth model in economics \cite{carmona2018probabilistic}, systemic risk models in finance \cite{carmona2013mean}, and cryptocurrency mining models in blockchain technology \cite{li2024mean}. While classical MFGs offer potential analytical traceability and dimension reduction in large-population settings, their validity depends on the assumption that all players are identical in distribution, which may not hold in many practical scenarios. Graphon MFGs address this limitation by modeling structured interactions, but they rely on the assumption that the underlying network is dense. However, many real-world networks, including social networks, are often sparse, where graphon-based approximations may be less effective.
In contrast, our setting directly models finite-agent games, accommodating both sparse and dense interactions without requiring a large game size. For further research on learning graphon mean-field games, we refer readers to \cite{cui2021learning,fabian2023learning}.

In this paper, we propose a new finite-agent linear-quadratic game on graphs, aiming to provide a more realistic modeling of players with heterogeneous and interpretable interactions. 
More precisely, we consider an $N$-player game, where the dynamics of player $i$ are given by
\begin{equation}
    \ud X^i_t = \left[a\left(\bar X_t^i - X^i_t\right) + \alpha^i_t\right]\,\ud t + \sigma\,\ud W^i_t, \quad \forall i\in[N],
\end{equation} 
where $\alpha^i$ is player $i$'s strategy and $\bar X_t^i := \frac{1}{\sqrt{d_{v_i}}}\sum_{j:v_j\sim v_i}\frac{1}{\sqrt{d_{v_j}}}X^j_t$ denotes the local average.
Here $d_{v_i}$ represents the degree of node $i$ in the graph, characterizing the connectivity among $N$ players. The dynamics exhibit a mean-reverting nature, where the reversion level is determined by the state average of neighboring players. One key motivation for this model is peer influence in social networks, where $X_t^i$ represents player $i$'s opinion and the dynamics reflect opinion conformity. 

This framework extends existing research, including \cite{carmona2013mean}, which primarily focuses on games on equally-weighted complete graphs (where $d_{v_i} \equiv N - 1$, $\forall i$), and \cite{lacker2022case}, which assumes $a = 0$ thus does not include mean reversion.  By incorporating local interactions through the graph structure, our model captures heterogeneous influences and provides a more general setting for strategic decision-making. For further details of the model, including motivation and interpretation, see Section~\ref{sec:model_setup}, particularly Remark~\ref{rem:model}. 

\medskip
\noindent{\textbf{Main Contribution.}}
In Section~\ref{sec:fp}, we analyze games on arbitrary graphs. Using dynamic programming and a properly chosen ansatz, we translate the game into a system of $N$ coupled matrix-valued Riccati equations, whose solvability has not yet been fully characterized in the literature. We establish local existence and uniqueness of the solution under a smallness condition, regardless of the number of players $N$. This is a key improvement over previous results such as \cite{papavassilopoulos1979existence}, where the existence condition depends explicitly on $N$. 
Our approach leverages the convergence of fictitious play to solve for the Markovian NE, which extends the previous analysis \cite{hu2019deep} that only addresses the case of open-loop NE for games with homogeneous players on complete graphs.
In addition, our work provides theoretical backing for the feasibility of FP.

In Section~\ref{sec:transitive}, we focus on games on vertex-transitive graphs. The added symmetry among players allows us to refine the solution obtained from the Riccati system and construct a semi-explicit characterization of the NE. This characterization is intricately linked to solving a matrix-valued ODE, whose existence and uniqueness are far from trivial. Our key results include local well-posedness (Theorem~\ref{thm:local_exist_unique}), established via the Picard-Lindel\"of theorem, and global existence and uniqueness (Theorem~\ref{thm:global_exist_unique}), obtained through a concavity argument that provides an \textit{a priori} upper bound.

For comparison, in \cite{carmona2013mean}, the Riccati equations are scalar-valued and decoupled, allowing for explicit solutions. In \cite{lacker2022case}, where \(a= 0\), the ODE degenerates to a scalar equation, whose well-posedness follows from classical theory.

\medskip
\noindent \textbf{Organization of the paper.} In the rest of the paper, we present multiple results from different perspectives, uncovering the connection between graph structures and the stochastic differential games on them. 
In Section~\ref{sec:model}, we introduce the game on graphs to be investigated, derive the Hamilton-Jacobi-Bellman (HJB) system via the dynamic programming principle, and reduce it into coupled Riccati equations as a preparation for finding the Markovian Nash equilibrium. 
Section~\ref{sec:fp} analyzes the iterative scheme -- fictitious play -- for solving the Ricatti equations and proves its convergence. 
We restrict ourselves to vertex-transitive graphs in Section~\ref{sec:transitive}. This class of graphs exhibits symmetry, from which we derive a semi-explicit expression of the Markovian Nash equilibrium, along with the development of the global existence and uniqueness result.
Section~\ref{sec:numerics} presents numerical experiments solving linear-quadratic games on different graphs, providing verification for the theorems proved in the previous sections and crucial insights into the connection between various aspects of the game and the graph structure. The paper concludes in Section~\ref{sec:conclusion_future}.


\section{A Linear-Quadratic Game on Graphs}\label{sec:model}

In this section, we consider an \(N\)-player linear-quadratic game associated with a finite graph \(G = (V,E)\) that is connected, simple, and has undirected edges. 
When associated with games, each vertex in the graph $G$ can be interpreted as a player, while each edge in the graph is identified as the interaction between two players. 

We start by fixing some notations and definitions. Denote \([N]\) as the collection of integers \(\SET{1,2,\ldots,N}\) and \(\mathbb{S}^{N\times N}\) as the collection of symmetric \(N\times N\) matrices. 
The graph-related concepts are defined as follows. 

\begin{defn}
Let \(G = (V,E)\) be a finite connected graph with simple undirected edges, where \(V = \SET{v_1,v_2,\ldots,v_N}\) denotes the collection of vertices, and each edge \(e\in E\) is an unordered pair \(e = (u,v)\) for some \(u,v\in V\).
Two vertices \(u,v\in V\) are said to be adjacent \(u\sim v\) if and only if \((u,v)\in E\). The degree of the vertex \(v\in V\), denoted by \(d_{v}\), is the number of edges that are connected to $v$.  
    The neighborhood of a vertex \(v\in V\) is defined as the collection of vertices that are one-step away from \(v\), i.e., \(N_G(v) = \SET{u\in V:u\sim v}\).
    \label{defn:graph}
\end{defn}

In this paper, the connection between a stochastic differential game and a graph is established through the graph Laplacian \(L\),
which, validated by the spectral graph theory \cite{chung1997spectral}, reveals crucial graph properties, e.g., graph connectivity, graph isoperimetric property, etc.

\begin{defn}[Normalized graph Laplacian]
    For a connected simple undirected graph \(G\) with \(N\) vertices, the (normalized) graph Laplacian \(L\in\mathbb{S}^{N\times N}\) is defined as
    \begin{equation}
    L_{ij} = 
    \begin{cases}
        1 & \text{if}\ i = j\\
        -\frac{1}{\sqrt{d_{v_i}d_{v_j}}} & \text{if}\ i\neq j \ \text{and}\ v_i\sim v_j\\
        0 & \text{else}
    \end{cases},\ \forall i,j\in[N].
\end{equation}
\label{defn:graph_lap}
\end{defn}

\subsection{The game setup}\label{sec:model_setup}

Consider a filtered probability space \((\Omega,\F,\{\F_t\}_{t\geq 0},\PP)\), supporting independent Brownian motions \(\{W^1_t\},\ldots,\{W^N_t\}\) and \(\F_t = \sigma(W^i_t,\forall i\in[N])\). The state process of player \(i\), denoted by \(\{X^i_t\}\), is controlled through her strategy process \(\{\alpha_t^i\}\) and subject to idiosyncratic noises modeled by \(\{W^i_t\}\): 
\begin{equation}
    \ud X^i_t = \left[a\left(\frac{1}{\sqrt{d_{v_i}}}\sum_{j:v_j\sim v_i}\frac{1}{\sqrt{d_{v_j}}}X^j_t - X^i_t\right) + \alpha^i_t\right]\,\ud t + \sigma\,\ud W^i_t, \quad \forall i\in[N].
    \label{eqn:state_dynamics}
\end{equation} 
Here, the parameter \(a\geq 0\) models the speed of mean reversion, and the parameter \(\sigma>0\) indicates the volatility of the state dynamics.  For further interpretation and motivation of the model, see Remark~\ref{rem:model}. Without loss of generality, the state and strategy processes are both assumed to take values in \(\R\).

The generic model of our concern operates on the finite time horizon \([0,T]\). Player $i$ chooses her strategy $\{\alpha_t^i\}_{t \in [0, T]}$ within the admissible set \(\A\): 
\begin{equation}\label{def:A}
   \A := \left\{\alpha:\alpha\ \text{is progressively measurable w.r.t.}\ \{\F_t\},\ \E \int_0^T|\alpha_t|^2\ud t<\infty\right\}, 
\end{equation}
to minimize her expected cost of the form:
\begin{equation}\label{def:J}
    J^i(\alpha) := \mathbb{E}\left[\int_0^T f^i(t,  X_t, \alpha_t^i) \ud t + g^i(X_T)\right].
\end{equation}
The notations $X_t := [X_t^1, \ldots, X_t^N]\transpose$ and $\alpha_t := (\alpha_t^1, \ldots, \alpha_t^N)$ represent the state processes and the strategies of all players.
Under the notation of \(x = [x^1, \ldots, x^N]\transpose \in \RR^N \) as the state variable and \(\alpha  \in \RR\) as the control variable, the running and terminal costs of player \(i\in[N]\) are set as
\begin{align}   \label{eqn:running_cost}
f^i(t,x,\alpha) &= \frac{1}{2}(\alpha)^2 - q\alpha \Big[ \frac{1}{\sqrt{d_{v_i}}}\sum_{j:v_j\sim v_i}\frac{1}{\sqrt{d_{v_j}}}x^j - x^i\Big] \notag\\
&+ \frac{\EPS}{2}\Big[ \frac{1}{\sqrt{d_{v_i}}}\sum_{j:v_j\sim v_i}\frac{1}{\sqrt{d_{v_j}}}x^j - x^i\Big]^2,\\
    g^i(x) &= \frac{c}{2}\Big[ \frac{1}{\sqrt{d_{v_i}}}\sum_{j:v_j\sim v_i}\frac{1}{\sqrt{d_{v_j}}}x^j - x^i\Big]^2.
    \label{eqn:terminal_cost}
\end{align}
The parameters are in the range \(q\geq 0\), \(\EPS\geq 0\), \(c>0\), and the condition \(q^2\leq\EPS\) is required to ensure the non-negativity of \(f^i\) and the joint convexity of \(f^i\) in $(x, \alpha)$, which guarantees the well-posedness of the problem.

The model considered encompasses several models from the literature as special cases. For instance, when \(G = K_N\), a complete graph with \(N\) vertices, our model aligns with the one in \cite{carmona2013mean}; and when $a = q= \EPS =  0$ and $G$ is a vertex-transitive graph,  our model aligns with the one in \cite{lacker2022case}. 

\medskip

\begin{rem}[Model interpretation]\label{rem:model}
Using Definition~\ref{defn:graph_lap}, the mean reversion level in~\eqref{eqn:state_dynamics} can be rewritten as
\begin{equation}
    \frac{1}{\sqrt{d_{v_i}}}\sum_{j:v_j\sim v_i}\frac{1}{\sqrt{d_{v_j}}}X^j_t  - X^i_t = -e_i^TLX_t,
    \label{eqn:mean_rev_graph_lap}
\end{equation}
where \(\{e_1,\ldots,e_N\}\) is the standard basis of \(\R^N\).
This concise representation facilitates the connection between the graph structure and various aspects of stochastic differential games on graphs in later discussions.

One motivation for studying the dynamics  \eqref{eqn:state_dynamics} comes from modeling peer influence in a given social network \(G\). Here, \(X^i_t\) represents the opinion of player \(i\) on a certain topic, and the interactions occur only between neighboring nodes (friends). Under opinion conformity (where an individual’s opinion tends to align with those of others), the opinion dynamics naturally follow a mean-reversion process within each graph neighborhood. Players can actively adjust their opinions via taking actions \(\alpha^i_t\), incurring costs for deviations and exertion of effort. 
The cost functionals \eqref{eqn:running_cost}--\eqref{eqn:terminal_cost} encode the trade-off: each player seeks to conform to the average opinion of their immediate neighbors while minimizing adjustment costs.

We note that different definitions of the graph Laplacian exist in the graph theory literature.
The unnormalized graph Laplacian \(L^{\text{UN}}\) is a symmetric matrix with real eigenvalues \cite{godsil2001algebraic}.
However, its eigenvalues admit possible dependence on \(N\), and its spectral properties are dominated by high-degree graph vertices.
Among normalized graph Laplacian, one commonly used variant is the random walk normalized graph Laplacian \(L^{\text{RW}}\) defined in~\eqref{eqn:L_RW}, which is motivated by the transition probability of random walks on graphs:
\begin{equation}
    \label{eqn:L_RW}
    L^{\text{RW}}_{ij}:= 
    \begin{cases}
        1 & \text{if}\ i = j\\
        -\frac{1}{d_{v_i}} & \text{if}\ i\neq j \ \text{and}\ v_i\sim v_j\\
        0 & \text{else}
    \end{cases},\ \forall i,j\in[N].
\end{equation}
The mean reversion level associated with \(L^{\text{RW}}\) corresponds to the arithmetic average within the graph neighborhood:
\begin{equation}
    \frac{1}{d_{v_i}}\sum_{j:v_j\sim v_i}X^j_t  - X^i_t = -e_i^TL^{\text{RW}}X_t.
\end{equation}
Nevertheless, the asymmetry of \(L^{\text{RW}}\) on general graphs introduces technical challenges and complicates theoretical analysis.
To overcome these difficulties, we base our discussion on the symmetric normalized graph Laplacian \(L\), as defined in Definition~\ref{defn:graph_lap}.
A straightforward calculation shows that \(L^{\text{RW}}\) and \(L\) are similar matrices, related by \(L^{\text{RW}} = D^{-\frac{1}{2}}LD^{\frac{1}{2}}\), where \(D:= \mathrm{diag}(d_{v_1},\ldots,d_{v_N})\) is the diagonal degree matrix.
Thus, \(L\) can be understood as a symmetric version of \(L^{\text{RW}}\) that maintains all its spectral information.

While the above discussion motivates our Laplacian-based mean reversion formulation in equation~\eqref{eqn:mean_rev_graph_lap},
it is worth noting that the results in Section~\ref{sec:fp} extend to a broader class of averaging schemes beyond Laplacian-based ones, i.e., for different choices of symmetric matrices $L$, as long as \(L\) has a spectrum bounded by a constant. However, in Section~\ref{sec:transitive}, where we focus on games over vertex-transitive graphs, a high degree of symmetry is required, making the normalized Laplacian-based formulation particularly natural. Therefore, we choose to adopt this specific formulation~\eqref{eqn:mean_rev_graph_lap} throughout the following context.

Note that if \(G\) is a regular graph, meaning all vertices have the same degree, the mean reversion level in~\eqref{eqn:mean_rev_graph_lap} becomes the arithmetic average of the states of all players in the neighborhood of \(v_i\).

\end{rem}

\medskip

\begin{rem}

Although \(G\) is required to be a connected graph in Definition~\ref{defn:graph}, this model also works for disconnected graphs since one can always put up the same model on each connected component of a disconnected graph.
\label{rem:disconn}
\end{rem}

\medskip

In a stochastic differential game, the information set based on which each player makes its decision must be specified. The difference in the information set results in different notions of strategy, e.g., open-loop strategy, closed-loop strategy, Markovian strategy, etc. In this paper, our discussion is restricted to the Markovian strategy, i.e., \(\alpha^i_t = \phi^i(t,X_t)\) for some deterministic feedback function \(\phi^i:[0,T]\times \R^N\to\R\) of player \(i\). 
To match the restrictions in the definition of \(\A\) (cf. \eqref{def:A}), the feedback function \(\phi^i\) shall be Borel measurable such that \(\sup_{(t,x)\in[0,T]\times \R^N}\frac{|\phi^i(t,x)|}{1+\NORM{x}}<\infty\). In the setting of a competitive game, the Nash equilibrium defined below is of general interest.

\begin{defn}[Nash equilibrium]
    The collection of all players' strategies \(\hat{\alpha} := (\hat\alpha^1, \ldots, \hat\alpha^N) \) is defined as an Nash equilibrium (NE) if 
    \begin{align}
        J^i((\alpha,\hat{\alpha}^{-i}))\geq J^i(\hat{\alpha}),\ \forall i\in[N],\ \forall \alpha \in \A.
    \end{align}
    The notation \((\alpha,\hat{\alpha}^{-i}) := (\hat{\alpha}^1, \ldots, \hat{\alpha}^{i-1}, \alpha,\hat{\alpha}^{i+1}, \ldots ,\hat{\alpha}^N)\)
    is the collection of strategies replacing player \(i\)'s strategy with \(\alpha\) while maintaining all other players' strategies.
    \label{defn:NE}
\end{defn}

In simple words, a Nash equilibrium is a collection of strategies from which any player does not have the motivation to deviate, given that all other players stick to their Nash equilibrium strategies.

\subsection{The associated HJB system}

A standard approach to solving the Markovian NE involves constructing a value function and the Hamilton-Jacobi-Bellman (HJB) system.
Denote by \(v^i(t, x):[0,T]\times \R^N\to \R\) the value function of player \(i\),  where \(t\in [0,T]\) is the time variable and \(x \in\R^N\) is the state variable.
By dynamic programming principle, the HJB system is given by
\begin{multline}
    \partial_t v^i + \inf_{\alpha^i}\Big\{\sum_{j=1}^N (\alpha^j - ae_j\transpose Lx)\partial_{x^j}v^i + \frac{\sigma^2}{2} \sum_{j=1}^N \partial_{x^jx^j}v^i \\
    + \frac{1}{2}(\alpha^i)^2 + q\alpha^ie_i\transpose Lx + \frac{\EPS}{2}(e_i\transpose Lx)^2 \Big\} = 0,
    \label{eqn:HJB_system}
\end{multline}
with a terminal condition \(v^i(T,x) = \frac{c}{2}(e_i\transpose Lx)^2\). Solving for the infimum in the above equation gives
\begin{equation}
    \hat{\alpha}^i(t,x) = -qe_i\transpose Lx - \partial_{x^i}v^i(t,x).
    \label{eqn:NE_strategy_with_value_function}
\end{equation}
Using a quadratic (in $x$) ansatz, 
\begin{equation}
    v^i(t,x) = \frac{1}{2}x\transpose F^i_tx + h^i_t,
    \label{eqn:ansatz}
\end{equation}
where \(F^i:[0,T]\to \mathbb{S}^{N\times N}\) and \(h^i:[0,T]\to\R\) are both deterministic and measurable functions in time.
Plugging equations~\eqref{eqn:NE_strategy_with_value_function}--\eqref{eqn:ansatz} back into equation~\eqref{eqn:HJB_system}, and collecting the coefficients of \(x\) yield a Riccati equation for \(F^i\):
\begin{multline}
    \dot{F}^i_t + \sum_{j=1}^N [-(a+q)L - F^j_t]e_je_j\transpose F^i_t + \sum_{j=1}^N F^i_te_je_j\transpose [-(a+q)L - F^j_t] \\
    +(\EPS-q^2)Le_ie_i\transpose L + F^i_te_ie_i\transpose F^i_t = 0,\quad F^i_T = cLe_ie_i\transpose L.
    \label{eqn:Riccati}
\end{multline}
By collecting terms of order 1, one can deduce an ODE for $h^i$:
\begin{equation}
    \dot{h}^i_t + \half \sigma^2 \Tr(F^i_t) = 0, \quad h^i_T = 0,
    \label{eqn:Riccati_for_h}
\end{equation}
which is in first order involving $F^i$. Therefore, \(h^i\) can be easily represented by \(F^i\), while the main difficulty lies in solving for \(F^i\). 
If a solution \((F^1_t,\ldots,F^N_t)\) to the coupled Riccati system~\eqref{eqn:Riccati} exists, a corresponding Markovian NE \(\hat{\alpha}\) exists, which is given by 
\begin{equation}
    \hat{\alpha}^i(t,x) = -qe_i\transpose Lx - e_i\transpose F^i_tx.
    \label{eqn:NE_strategy}
\end{equation}
This equilibrium strategy depends on $F^i$ but not on $h^i$.
Consequently, the rest of the paper will focus on analyzing the Riccati system~\eqref{eqn:Riccati}.


\section{The Equilibrium via Fictitious Play}\label{sec:fp}

Initially introduced by Brown \cite{brown1949some,brown1951iterative}, the key concept of FP involves transforming the calculation of the NE into a sequence of stochastic control problems faced by each player.
Specifically, player $i$ optimizes her objective while assuming all other players' strategies remain fixed and follow their historical strategies. 
Solving this stochastic control problem yields a new optimal strategy for player \(i\), completing one round of fictitious play for player \(i\).
Each round of fictitious play consists of performing this procedure for each player once. The ultimate goal is to observe the convergence of all players' strategies converging to the NE after a sufficient number of rounds of fictitious play have been carried out.

The convergence of FP to the NE, however, is not guaranteed, even for tabular games \cite{shapley1963some}, which motivates research on the convergence of FP. In the context of continuous-time differential games, affirmative answers include  \cite{cardaliaguet2017learning,briani2018stable} for mean-field games and \cite{hu2019deep,han2022convergence} for finite-player games, but none of them take into account the graph structure. 
In this section, we provide proof for the convergence of FP for the generic model described in Section~\ref{sec:model_setup}.
This result contributes in three significant ways. Firstly, it provides a method for constructing the NE under conditions that are independent of the game size $N$. 
Secondly, it validates the effectiveness of FP for solving stochastic differential games on graphs. 
Thirdly, it elucidates the connection between the graph structure and the convergence of FP.

\subsection{The mathematical formulation}

To prepare for stating the convergence result in later contexts, we describe the procedure of FP in mathematical terms.
Firstly, rewrite the state dynamics and the cost functionals in terms of the graph Laplacian \(L\). 
Using equation~\eqref{eqn:mean_rev_graph_lap}, the state dynamics~\eqref{eqn:state_dynamics} can be rewritten as
\begin{equation}
    \ud X^i_t = \left(\alpha^i_t - ae_i\transpose LX_t\right)\,\ud t + \sigma\,\ud W^i_t,
    \label{eqn:state_dynamics_L}
\end{equation}
and the cost functionals~\eqref{eqn:running_cost}--\eqref{eqn:terminal_cost} are rewritten as
\begin{equation}
    f^i(t,x,\alpha) = \frac{1}{2}(\alpha)^2 + q\alpha e_i\transpose Lx + \frac{\EPS}{2}\left( e_i\transpose Lx\right)^2,
    \quad g^i(x) = \frac{c}{2}\left( e_i\transpose Lx\right)^2.
    \label{eqn:cost_L}
\end{equation}
To provide a clear description of how FP works, we summarize the notations used as follows. 

\begin{nota}
    For any object \(A\), \(A^{i,k}\) refers to the object associated with player \(i\) at FP stage \(k\in\N\).
    Denote by \(A^{k}\) the collection of objects across all the players at FP stage \(k\). 
    In the following context, \(i\) and \(j\) stand for the player indices, while \(k\) represents the FP stage index.
    \label{nota: FP}
\end{nota}

Without loss of generality, let us consider the procedure of FP for player \(i\).
Assume all players have their strategies at stage \(k\) determined through feedback functions \(\phi^{j,k}\) that \(\alpha^{j,k}_t = \phi^{j,k}(t,X_t^{k})\), $\forall j \in [N]$. 
At stage $k+1$, player \(i\) optimizes her strategy to get a new feedback function \(\phi^{i,k+1}\) based on the environment at stage \(k\) while assuming all the other players stick to their feedback functions at stage \(k\), i.e., player \(j\neq i\) is using feedback function \(\phi^{j,k}\).
As a result, the stochastic control problem faced by player \(i\) has state dynamics as
\begin{equation}
    \begin{cases}
    \ud X^{i,k+1}_t = (\alpha^i_t - ae_i\transpose LX^{k+1}_t)\,\ud t + \sigma\,\ud W^i_t,\\
    \ud X^{j,k+1}_t = [\phi^{j,k}(t,X^{k+1}_t) - ae_i\transpose LX^{k+1}_t]\,\ud t + \sigma\,\ud W^j_t,\ \forall j\neq i.
    \end{cases}.
    \label{eqn:FP_state_dynamics_tmp}
\end{equation}
Player \(i\) faces a running cost of \(f^i(t, X^{k+1}_t,\alpha^i_t)\) and terminal cost \(g^i(X_T^{k+1})\) where \(f^i\) and \(g^i\) are provided in equation~\eqref{eqn:cost_L}.

\begin{rem}[Variants of FP algorithms]
    
The original fictitious play (FP) \cite{brown1949some, brown1951iterative} updates strategies based on the best response to the empirical distribution of opponents' past actions. It has been shown to be an effective method for converging to a NE in normal-form games \cite{robinson1951an,hahn1999the,berger2005fictitious}.

In mean-field games (MFGs), FP has been adapted to update strategies based on the average density/measure of past population distributions \cite{briani2018stable, cardaliaguet2017learning, min2021signatured}. An alternative approach, which we adopt here, updates strategies using only the most recent interaction. This method has been employed in both MFGs \cite{han2024learning} and finite-player games \cite{hu2019deep, han2020deep, han2022convergence, xuan2022pandemic}. The primary motivation behind this choice is computational and due to memory efficiency---averaging strategies is often infeasible when strategies are parameterized by neural networks, since computing empirical averages introduces additional computational costs.

All the FP variants mentioned above follow simultaneous FP, meaning that the $(k+1)^{th}$ strategies for all players can be computed at the same time. In contrast, alternating FP updates strategies sequentially in a specific order. A thorough discussion of these variants can be found in \cite[Sections~5.2--5.4]{hu2019deep}.

\end{rem}

\medskip

Let \(v^{i,k+1}:[0,T]\times \R^N\to \R\) be the value function of player \(i\) at stage \(k+1\). It satisfies 
\begin{multline}
    \partial_t v^{i,k+1} + \inf_\alpha\Big\{(\alpha - ae_i\transpose Lx)\partial_{x^i}v^{i,k+1} + \sum_{j\neq i}(\phi^{j,k}(t,x) - ae_j\transpose Lx)\partial_{x^j}v^{i,k+1} \\
    + \frac{\sigma^2}{2}\sum_{j=1}^N\partial_{x^jx^j}v^{i,k+1}
    + \frac{1}{2}\alpha^2 + q\alpha e_i\transpose Lx + \frac{\EPS}{2}(e_i\transpose Lx)^2\Big\} = 0, 
    \label{eqn:HJB_FP}
\end{multline}
with terminal condition $ v^{i,k+1}(T,x) = \frac{c}{2}(e_i\transpose Lx)^2$.
The first-order condition in  the above equation gives
\begin{equation}
    \phi^{i,k+1}(t,x) = -\partial_{x^i}v^{i,k+1}(t,x) - qe_i\transpose Lx.
    \label{eqn:opt_ctrl_FP}
\end{equation}
Plugging the above feedback function back into equation~\eqref{eqn:HJB_FP} yields
\begin{multline}
    \partial_t v^{i,k+1} - (a+q)e_i\transpose Lx\;\partial_{x^i}v^{i,k+1}
    + \frac{\sigma^2}{2}\sum_{j=1}^N \partial_{x^jx^j}v^{i,k+1} - \frac{1}{2}(\partial_{x^i}v^{i,k+1})^2 \\
    +\frac{\EPS - q^2}{2}(e_i\transpose Lx)^2
    + \sum_{j\neq i} \left[\phi^{j,k}(t,x) - ae_j\transpose Lx\right]\partial_{x^j}v^{i,k+1} = 0.
    \label{eqn:HJB_FP_simplified}
\end{multline}
The term \(\phi^{j,k}(t,x)\) in the above equation hinders us from reducing the HJB equation to a Riccati equation directly.
To proceed, we make the following assumption on the initial strategies $\phi^{j,0}$, for all players $j \in [N]$.

\begin{assu}
    Assume that for every player $j$, \(j\in[N]\), FP starts with \(\phi^{j,0}\) that is of the form \(\phi^{j,0}(t,x) = (\varphi^{j,0}_t)\transpose x\), for some \(\varphi^{j,0}:[0,T]\to\R^N\).
    \label{assu:feedback}
\end{assu}

\begin{lem}\label{lem:FP}
    Under Assumption~\ref{assu:feedback}, the optimal strategy $\phi^{i,k}$ for player $i$ at FP stage $k$, can always be represented as \(\phi^{i,k}(t,x) = (\varphi^{i,k}_t)\transpose x\) for some \(\varphi^{i,k}:[0,T]\to\R^N\),
    for any \( k\in\N\).
\end{lem}

\begin{proof}
    We present a proof by induction on $k$. Let us assume that the statement holds at stage $k$, and consider a quadratic (in $x$) ansatz for the value function $v^{i, k+1}(t,x)= \frac{1}{2}x\transpose F^{i,k+1}_t x + h^{i,k+1}_t$,
where \(F^{i,k+1}:[0,T]\to \mathbb{S}^{N\times N}\) and \(h^{i,k+1}:[0,T]\to \R\) are both deterministic functions in time.
By plugging it into equation~\eqref{eqn:HJB_FP_simplified} and using the induction hypothesis, one can collect terms of $x$ and obtain a Riccati equation for $F^{i, k+1}$:
\begin{multline}
    \dot{F}^{i,k+1}_t - F^{i,k+1}_te_ie_i\transpose F^{i,k+1}_t + (\EPS-q^2)Le_ie_i\transpose L - aLF^{i,k+1}_t - aF^{i,k+1}_tL \\
    + \sum_{j\neq i}F^{i,k+1}_t e_j(\varphi^{j,k}_t)\transpose
    + \sum_{j\neq i} \varphi^{j,k}_te_j\transpose F^{i,k+1}_t - qLe_ie_i\transpose F^{i,k+1}_t - qF^{i,k+1}_t e_ie_i\transpose L = 0,
    \label{eqn:Riccati_FP_varphi}
\end{multline}
with terminal condition $F^{i,k+1}_T = cLe_ie_i\transpose L$.

Plugging  the ansatz into equation~\eqref{eqn:opt_ctrl_FP} yields \(\phi^{i,k+1}(t,x) = -e_i\transpose (F^{i,k+1}_t + qL)x\). Therefore, \(\forall k\in\N\),
\begin{equation}
    \varphi^{i,k+1}_t = -(F^{i,k+1}_t + qL)e_i.
    \label{eqn:varphi_F_FP}
\end{equation}
This completes the proof. 
\end{proof}

This lemma indicates that if the FP procedure begins (at stage 0) with linear strategies (in $x$) for all players, it will result in every strategy being linear in $x$ for all subsequent stages. 

Note that $\{F^{j,k+1}\}_{j \in [N]}$ and $\{F^{j,k}\}_{j \in [N]}$ are coupled through the term \(\varphi^{j,k}_t\), represented as \(\varphi^{j,k}_t = -(F^{j,k}_t + qL)e_j\) (cf. \eqref{eqn:varphi_F_FP}).
This yields the following recursive Riccati equations, which describe the update from $\{F^{j,k}\}_{j \in[N]}$ to $F^{i,k+1}$:
\begin{multline}
    \dot{F}^{i,k+1}_t - F^{i,k+1}_te_ie_i\transpose F^{i,k+1}_t + (\EPS-q^2)Le_ie_i\transpose L - (a+q)LF^{i,k+1}_t - (a+q)F^{i,k+1}_tL \\
    - F^{i,k+1}_t\sum_{j\neq i} e_je_j\transpose F^{j,k}_t - \sum_{j\neq i} F^{j,k}_te_je_j\transpose F^{i,k+1}_t = 0,\quad
    F^{i,k+1}_T = cLe_ie_i\transpose L.
    \label{eqn:recursive_Riccati_FP}
\end{multline}
Therefore, the FP procedure for the game introduced in Section~\ref{sec:model_setup} has been completely described by system~\eqref{eqn:recursive_Riccati_FP} following from Assumption~\ref{assu:feedback}.
In Proposition~\ref{prop:recursive_Riccati_welldefined} we prove that, given \(\{F^{j,k}\}_{j\in[N]}\), system~\eqref{eqn:recursive_Riccati_FP} uniquely determines \(\{F^{j,k+1}\}_{j\in[N]}\) on \([0,T]\) for any \(T>0\), ensuring that the FP iteration scheme is well-defined.

Unless otherwise specified, all matrix norms are understood as the matrix \(2\)-norm, and inequalities between symmetric matrices are understood in the positive semi-definite sense. That is,  for \(A,B\in\mathbb{S}^{N\times N}\), we write \(A\geq B\) if and only if \(A-B\) is positive semi-definite.

\begin{prop}
    \label{prop:recursive_Riccati_welldefined}
    Given \(\{F^{j,0}\}_{j\in[N]}\) such that \(F^{j,0}_t\) is bounded on \([0,T]\) for any \(j\in[N]\), system~\eqref{eqn:recursive_Riccati_FP} admits a unique solution \(\{F^{j,k+1}\}_{j\in[N]}\) on \([0,T]\) for any \(T>0\) and \(k\in\N\).
    Moreover, \(F^{j,k+1}_t\geq 0\) for any \(t\in[0,T]\), \( k\in\N\), \(j\in[N]\).
\end{prop}

\begin{proof}
    We prove the result by induction on $k$. Suppose that \(\{F^{j,k}\}_{j\in[N]}\) is given.
    Define the function \(f:[0,T]\times \mathbb{S}^{N\times N}\to \mathbb{S}^{N\times N}\) such that
    \begin{multline}
        f(t,X;\{F^{j,k}\}_{j\in[N]}) :=  Xe_ie_i\transpose X - (\EPS-q^2)Le_ie_i\transpose L + (a+q)LX + (a+q)XL \\
        + X\sum_{j\neq i} e_je_j\transpose F^{j,k}_t + \sum_{j\neq i} F^{j,k}_te_je_j\transpose X.
    \end{multline}
    Then system~\eqref{eqn:recursive_Riccati_FP} can be rewritten as \(\dot{F}^{i,k+1}_t = f(t,F^{i,k+1}_t;\{F^{j,k}\}_{j\in[N]})\) with the terminal condition \(F^{i,k+1}_T = cLe_ie_i\transpose L\).
    
    We claim that: \(f\) is Lipschitz continuous in \(X\) if there exists \(M>0\) such that \(\NORM{X}\leq M\).
    This follows directly from the triangle inequality:
    \begin{align}
        \NORM{f(t,X_1) - f(t,X_2)}\leq 2\left(M + (a+q)\NORM{L} + \NORM{\sum_{j\neq i}e_je_j\transpose F^{j,k}_t}\right)\NORM{X_1-X_2}.
    \end{align}
    Since \(f\) is continuous in \(t\), the Picard-Lindelöf theorem \cite{teschl2012ordinary} ensures that system~\eqref{eqn:recursive_Riccati_FP} has a unique local solution at time \(T\).
    
    To achieve the global existence and uniqueness result, it suffices to establish the \textit{a priori} boundedness of \(F^{i,k+1}_t\).
    Consider the auxiliary ODE system for \(\{G^{j,k+1}\}_{j\in[N]}\):
    \begin{multline}
    \dot{G}^{i,k+1}_t + (\EPS-q^2)Le_ie_i\transpose L - (a+q)LG^{i,k+1}_t - (a+q)G^{i,k+1}_tL \\
    - G^{i,k+1}_t\sum_{j\neq i} e_je_j\transpose F^{j,k}_t - \sum_{j\neq i} F^{j,k}_te_je_j\transpose G^{i,k+1}_t = 0,\quad
    G^{i,k+1}_T = cLe_ie_i\transpose L.
    \label{eqn:auxiliary_G}
    \end{multline}
    Since this is a linear ODE system with bounded time-variant coefficients, it admits a unique bounded solution \(\{G^{j,k+1}\}_{j\in[N]}\) on \([0,T]\) for any \(T>0\).
    Consider the difference \(D^{i,k+1}_t := F^{i,k+1}_t - G^{i,k+1}_t\), which satisfies \(D^{i,k+1}_T = 0\) and \(\dot{D}^{i,k+1}_t\geq 0\).
    Therefore, \(D^{i,k+1}_t\leq 0\), implying an upper bound \(F^{i,k+1}_t\leq G^{i,k+1}_t\), \(\forall t\in[0,T]\).

    For any \(j\in[N]\), \(k\in\N\),  the cost functionals $f^j$ and $g^j$ stay non-negative, so does the value function \(v^{j, k+1}(t,x)= \frac{1}{2}x\transpose F^{j,k+1}_t x + h^{j,k+1}_t\), \(\forall x\in\R^N\). Consequently, one obtains the lower bound \(F^{j,k+1}_t\geq 0\), \(\forall t\in[0,T]\).

    Combining both parts and setting \(M:=\sup_{t\in[0,T]}\sup_{j\in[N]}\NORM{G^{j,k+1}_t}\) conclude the proof.  
\end{proof}

\subsection{Convergence results}

In this section, we prove that the iterative scheme described by \eqref{eqn:recursive_Riccati_FP} converges to a limit, which is the Markovian NE. 
We first define a notion of convergence when discussing the convergence of strategies. Since the players' strategies are represented by feedback functions of state processes, we examine the pointwise convergence of the feedback function \(\phi^{i,k}(t,x)\) as \(k\to\infty\) for any fixed \(i\in[N]\).
As implied by Lemma~\ref{lem:FP} and equation~\eqref{eqn:varphi_F_FP}, the pointwise convergence of \(\phi^{i,k}\) is equivalent to the convergence of \(F^{i,k}\).
Therefore, it suffices to prove (in Theorem~\ref{thm:conv_FP}) that the solution \(F^{i,k}\) to system~\eqref{eqn:recursive_Riccati_FP} admits a limit as \(k\to\infty\) for any fixed \(i\in[N]\) , and the limit satisfies equation~\eqref{eqn:Riccati}. 
The proof relies on the following technical lemma as a matrix inequality.

\begin{lem}
    For symmetric positive semi-definite matrices \(A_1,\ldots,A_N\), the following inequality holds:
    \begin{equation}
        \NORM{\sum_{i=1}^N e_ie_i\transpose A_i}\leq \NORM{\sum_{i=1}^N A_i}.
    \end{equation}
    \label{lemma:matrix_norm_A_B}
\end{lem}

\begin{proof}
To simplify the notation, let  \(B_i:= \sum_{i=1}^N e_ie_i\transpose A_i\).
    Since \(\sum_{i=1}^N A_ie_ie_i\transpose A_i\leq \sum_{i=1}^N A_i^2\), it is clear
    \begin{equation}
        \NORM{B_i}^2 = \NORM{B_i\transpose B_i} = \NORM{\sum_{i=1}^N A_ie_ie_i\transpose A_i}\leq \NORM{\sum_{i=1}^N A_i^2}.
    \end{equation}
    Due to the matrix inequality \(\sum_{i=1}^N A_i^2\leq \sum_{i=1}^N \NORM{A_i}\;A_i\leq \max_{i\in[N]}\NORM{A_i}\;\sum_{i=1}^N A_i\), the following inequality holds:
    \begin{equation}
        \NORM{\sum_{i=1}^N A_i^2}\leq \max_{i\in[N]}\NORM{A_i}\;\NORM{\sum_{i=1}^N A_i}\leq \NORM{\sum_{i=1}^N A_i}^2,
    \end{equation}
    which concludes the proof.
\end{proof}

We next state the first main result of this section, concerning the convergence of $F^{i,k}$ as $k \to \infty$.

\begin{thm}
    For any function \(A^{i,k}:[0,T]\to \mathbb{S}^{N\times N}\), denote \(\Delta A^{i,k} = A^{i,k+1} - A^{i,k}\) as the difference from stage \(k\) to \(k+1\).

    Under Assumption~\ref{assu:feedback}, if the fictitious play starts from \(F^{i,0}\equiv 0\) for any \(i\in[N]\), and the following condition
    \begin{equation}
        \max\left\{[c + T(\EPS-q^2)]\;\NORM{L}^2, \frac{2}{\log 2}(a+q)T\NORM{L}\right\}\leq 1, \quad T \leq \frac{1}{100},
        \label{eqn:FP_condition}
    \end{equation}
    holds, then the scheme \eqref{eqn:recursive_Riccati_FP} produced by fictitious play converges to a Markovian NE for the linear-quadratic games on graphs~\eqref{eqn:state_dynamics}--\eqref{def:J}.
    \label{thm:conv_FP}
\end{thm}

\begin{proof}

\noindent \textbf{Step 1: The a priori bound for $\NORM{F^{i,k}_t}$.}
A change of variable in time \(W^{i,k}_t := F^{i,k}_{T-t}\) yields
\begin{multline}
    \dot{W}^{i,k+1}_t = - (W^{i,k+1}_t - W^{i,k}_t)e_ie_i\transpose (W^{i,k+1}_t - W^{i,k}_t) + (\EPS-q^2)Le_ie_i\transpose L  - (a+q)LW^{i,k+1}_t \\
    - (a+q)W^{i,k+1}_tL 
    - W^{i,k+1}_t\sum_{j=1}^N e_je_j\transpose W^{j,k}_t - \sum_{j = 1}^N W^{j,k}_te_je_j\transpose W^{i,k+1}_t + W^{i,k}_te_ie_i\transpose W^{i,k}_t,
\end{multline}
with terminal condition $W^{i,k+1}_0 = cLe_ie_i\transpose L$.
Using \(W^{i,k+1}_t = W^{i,k+1}_0 + \int_0^t \dot{W}^{i,k+1}_s\,\ud s\) and \((W^{i,k+1}_t - W^{i,k}_t)e_ie_i\transpose (W^{i,k+1}_t - W^{i,k}_t)\geq 0\), followed by summing both sides with respect to \(i\in[N]\), gives the following matrix inequality:
{\small\begin{multline}
    \sum_{i=1}^N W^{i,k+1}_t \leq [c+ t(\EPS - q^2)]L^2 - (a+q)L \int_0^t \sum_{i=1}^N W^{i,k+1}_s\,\ud s - (a+q) \int_0^t \sum_{i=1}^N W^{i,k+1}_s\,\ud s\;L\\
    - \int_0^t \sum_{i=1}^N W^{i,k+1}_s\sum_{j=1}^N e_je_j\transpose W^{j,k}_s\,\ud s - \int_0^t \sum_{j=1}^N  W^{j,k}_se_je_j\transpose\sum_{i=1}^N W^{i,k+1}_s\,\ud s + \int_0^t \sum_{i=1}^N W^{i,k}_se_ie_i\transpose W^{i,k}_s\,\ud s.
\end{multline}}
Let \(A^k_t := \sum_{j=1}^N W^{j,k}_t\) and \(B^k_t := \sum_{j=1}^N e_je_j\transpose W^{j,k}_t\).
The matrix inequality above is rewritten as
\begin{multline}
    A^{k+1}_t\leq [c+ t(\EPS - q^2)]L^2 - (a+q)L \int_0^t A^{k+1}_s\,\ud s - (a+q) \int_0^t A^{k+1}_s\,\ud s\;L\\
    - \int_0^t A^{k+1}_sB^k_s\,\ud s - \int_0^t (B^k_s)\transpose A^{k+1}_s\,\ud s + \int_0^t (B^k_s)\transpose B^k_s\,\ud s.
    \label{eqn:ineq_A_B}
\end{multline}
Since Proposition~\ref{prop:recursive_Riccati_welldefined} implies that \(W^{i,k}_t\geq 0\), \(\forall t\in[0,T]\),
Lemma~\ref{lemma:matrix_norm_A_B} tells \(\NORM{B^k_t}\leq \NORM{A^k_t}\).
Combining with inequality~\eqref{eqn:ineq_A_B}, we have
\begin{equation}
    \NORM{A^{k+1}_t}\leq [c+ t(\EPS - q^2)]\NORM{L}^2 + 2\int_0^t \left[(a+q)\NORM{L} + \NORM{A^k_s}\right]\;\NORM{A^{k+1}_s}\,\ud s + \int_0^t \NORM{A^k_s}^2\,\ud s.
\end{equation}
Then one can apply Grönwall's inequality to get:
\begin{equation}
    \NORM{A^{k+1}_t} \leq  \left([c+ t(\EPS - q^2)]\NORM{L}^2 + \int_0^t \NORM{A^k_s}^2\,\ud s\right)
    e^{2\int_0^t (a+q)\NORM{L} + \NORM{A^k_s} \,\ud s}.
\end{equation}
Taking the supremum on both sides with respect to \(t\in[0,T]\) yields
\begin{equation}
    \sup_{t\in[0,T]}\NORM{A^{k+1}_t} \leq  \bigg([c+ T(\EPS - q^2)]\;\NORM{L}^2 + T\sup_{t\in[0,T]}\NORM{A^k_t}^2\bigg)
    e^{2T (a+q)\NORM{L} + 2T\sup_{t\in[0,T]}\NORM{A^k_t}}.
\end{equation}
Define \(h:\R_+\to \R_+\) as a continuous increasing function:
\begin{equation}
    h(x) := \left([c+ T(\EPS - q^2)]\;\NORM{L}^2 + Tx^2\right)e^{2T (a+q)\NORM{L}}e^{2Tx}.
    \label{eqn:h_FP}
\end{equation}
Under condition~\eqref{eqn:FP_condition}, \([c+ T(\EPS - q^2)]\;\NORM{L}^2\leq 1\) and \(e^{2T (a+q)\NORM{L}}\leq 2\), so
\begin{equation}
    h(x)\leq 2\left(1 + 0.01x^2\right)e^{0.02x}:=g(x).
\end{equation}
Since \(g(2)>2\), \(g(3)<3\), by the intermediate value theorem, \(g\) has a fixed point \(x^* \in(2, 3)\).

Using the fixed point \(x^*\), we establish a uniform \textit{a priori} bound for \(\NORM{A^{k}_t}\). Specifically, 
we claim that \(\sup_{t\in[0,T]}\NORM{A^{k}_t}\leq x^*, \ \forall k\in\N\), which follows from straightforward induction argument.
For the base case $k = 0$, since FP starts from \(F^{i,0}\equiv 0\),  we have \(\sup_{t\in[0,T]}\NORM{A^{0}_t} = 0\leq x^*\) that confirms the claim. 
Now, assume the claim holds for \(k = l\), i.e., \(\sup_{t\in[0,T]}\NORM{A^{l}_t}\leq x^*\). By applying the recursive structure, we obtain
\begin{align}
    \sup_{t\in[0,T]}\NORM{A^{l+1}_t}\leq h\left(\sup_{t\in[0,T]}\NORM{A^{l}_t}\right)\leq g\left(\sup_{t\in[0,T]}\NORM{A^{l}_t}\right)\leq g(x^*) = x^*,
\end{align}
where the final inequality follows from the induction hypothesis \(\sup_{t\in[0,T]}\NORM{A^{l}_t}\leq x^*\) and the monotonicity of \(g\).
Consequently, we conclude that \(\NORM{\sum_{j=1}^N F^{j,k}_t}\leq x^*\), \(\forall t\in[0,T]\), \(\forall k\in\N\).
Combining with the fact that \(F^{i,k}_t\) takes values as positive semi-definite matrices proves the a priori bound \(\NORM{F^{j,k}_t}\leq x^*\), \(\forall k\in\N\), \(\forall t\in[0,T]\), \(\forall j\in[N]\).

\medskip

\noindent\textbf{Step 2: The uniform convergence.}
Using \(F^{i,k+1}_t = F^{i,k+1}_T - \int_t^T \dot{F}^{i,k+1}_s\,\ud s\) and replacing \(\dot{F}^{i,k+1}_s\) by equation~\eqref{eqn:recursive_Riccati_FP} yields the following upper bound:
\begin{multline}
    \NORM{\Delta F^{i,k}_t} \hspace{-2pt}
    \leq 2(a+q)\NORM{L}\int_t^T \NORM{\Delta F^{i,k}_s}\,\ud s
    + \int_t^T \left(\NORM{F^{i,k}_s} \NORM{\Delta F^{i,k}_s} + \NORM{F^{i,k+1}_s}\NORM{\Delta F^{i,k}_s}\right)\,\ud s\\
    + 2\int_t^T \left(\NORM{\Delta F^{i,k}_s} \sum_{j\neq i}\NORM{F^{j,k-1}_s} + \NORM{F^{i,k+1}_s}  \sum_{j\neq i}\NORM{\Delta F^{j,k-1}_s}\right)\,\ud s.
    \label{eqn:conv_FP_proof_tmp1}
\end{multline}
Taking the supremum on both sides with respect to \(i\in[N]\), and using the a priori bound \(\NORM{F^{i,k}_t}\leq x^*\) established in the last step, the following inequality is proved:
\begin{multline}
     \sup_{i\in[N]}\NORM{\Delta F^{i,k}_t}  
    \leq \left[2(a+q)\NORM{L} + 2x^* + 2Nx^*\right]\int_t^T \sup_{i\in[N]}\NORM{\Delta F^{i,k}_s}\,\ud s \\
    + 2Nx^*\int_t^T \sup_{i\in[N]}\NORM{\Delta F^{i,k-1}_s}\,\ud s.
\end{multline}
Define \(g_k(t) := e^{\beta t}\sup_{i\in[N]}\NORM{\Delta F^{i,k}_t}\) for some sufficiently large \(\beta >0\) such that \(\frac{2(a+q)\NORM{L} + 2x^* + 2Nx^*}{\beta}<\frac{1}{3}\).
It follows that
    \begin{equation}
        g_k(t)\leq [2(a+q)\NORM{L} + 2x^* + 2Nx^*] e^{\beta t}\int_t^T e^{-\beta s}[g_k(s) + g_{k-1}(s)]\,\ud s.
    \end{equation}
    Since \(g_k(t)\) is a continuous function on a compact domain, it satisfies \(\sup_{s\in[0,T]}g_k(s) < \infty\), which provides an upper bound for the integral above:
    \begin{equation}
        g_k(t)\leq \frac{2(a+q)\NORM{L} + 2x^* + 2Nx^*}{\beta}(1-e^{-\beta(T-t)}) \bigg(\sup_{s\in[0,T]}g_k(s) + \sup_{s\in[0,T]}g_{k-1}(s)\bigg).
    \end{equation}
    Taking the supremum with respect to \(t\in[0,T]\) on both sides yields the contraction
    \begin{equation}
        \label{eqn:contrac}
        \sup_{t\in[0,T]}g_k(t)\leq \frac{1}{2} \sup_{t\in[0,T]}g_{k-1}(t).
    \end{equation}
    This implies that \(\Delta F^k\) converges to zero under the norm \(\NORM{\Delta F^k}_\beta:= \sup_{t\in[0,T]}e^{\beta t}\sup_{i\in[N]}\NORM{\Delta F^{i,k}_t}\) as \(k\to\infty\).
    Since, on a finite time horizon \([0,T]\), the norm \(\NORM{\cdot}_\beta\) is equivalent to \(\NORM{\cdot}_0\), where \(\NORM{\Delta F^k}_0:= \sup_{t\in[0,T]}\sup_{i\in[N]}\NORM{\Delta F^{i,k}_t}\),
    this implies that \(F_t^{i,k}\) converges uniformly in \(t\in[0,T]\), for any fixed \(i\in[N]\), as \(k\to\infty\). 
    We denote the uniform limit by \(F^{i,\infty}_t := \lim_{k\to\infty}F^{i,k}_t\).

\medskip
\noindent\textbf{Step 3: The limit as the Nash equilibrium.}
    It remains to prove that \((F^{1,\infty}_t,\cdots,F^{N,\infty}_t)\), the limit of $(F^{1,k}_t,\cdots,F^{N,k}_t)$ as $k \to \infty$, is a solution to the coupled Riccati system~\eqref{eqn:Riccati}. Taking the limit as \(k\to\infty\) in equation~\eqref{eqn:recursive_Riccati_FP}, and it remains to justify the interchange of the differentiation with respect to \(t\) and the limit with respect to \(k\), i.e., 
    \(\lim_{k\to\infty}\dot{F}^{i,k}_t= \dot{F}^{i,\infty}_t\).
    
    Rewriting equation~\eqref{eqn:recursive_Riccati_FP}, we express
    $$\dot{F}^{i,k+1}_t = H(F^{i,k+1}_t,F^{1,k}_t,\cdots,F^{i-1,k}_t,F^{i+1,k}_t,\cdots,F^{N,k}_t),$$ 
    where \(H:(\mathbb{S}^{N\times N})^N\to \mathbb{S}^{N\times N}\) is a continuous function explicitly determined by the structure of equation~\eqref{eqn:recursive_Riccati_FP}. 
    Since \(\NORM{F^{j,k}_t}\leq x^*\) for all \(k\in\N\), \(j\in[N]\), and \(t\in[0,T]\), the domain of \(H\) is compact, and \(H\) is uniformly continuous.
    It is well-known that uniformly continuous transformations preserve uniform convergence.
    Since the convergence of \(F^{i,k}_t\) is uniform in \(t\in[0,T]\), it follows that \(\dot{F}^{i,k+1}_t\) also converges uniformly in \(t\in[0,T]\) as \(k\to\infty\). This justifies the interchange of the limit and differentiation, completing the proof.
\end{proof}

Lastly, we rescale the time horizon to relax condition~\eqref{eqn:FP_condition}, as shown in Corollary~\ref{cor:time_scaling}.
This concludes the investigation of the convergence of FP for our linear-quadratic games on graphs.

\begin{cor}
    Under Assumption~\ref{assu:feedback}, if the fictitious play starts from \(F^{i,0}\equiv 0\) for any \(i\in[N]\) and the following condition
    \begin{equation}
        C_L := \max\left\{100T[c + T(\EPS-q^2)]\;\NORM{L}^2, \frac{2}{\log 2}(a+q)T\NORM{L}\right\}\leq 1,
        \label{eqn:FP_final_condition}
    \end{equation}
    holds, then the scheme \eqref{eqn:recursive_Riccati_FP} produced by fictitious play converges to the Markovian NE for the linear-quadratic games on graphs~\eqref{eqn:state_dynamics}--\eqref{def:J}.
    \label{cor:time_scaling}
\end{cor}

\begin{proof}
    Define \(G^{i,k}_t := \frac{T}{T'}F^{i,k}_{\frac{T}{T'}t}\) for some positive \(T'\), equation~\eqref{eqn:recursive_Riccati_FP} becomes:
    \begin{multline}
        \dot{G}^{i,k+1}_t - G^{i,k+1}_te_ie_i\transpose G^{i,k+1}_t + \left(\frac{T}{T'}\right)^2(\EPS-q^2)Le_ie_i\transpose L - \frac{T}{T'}(a+q)LG^{i,k+1}_t \\
        - \frac{T}{T'}(a+q)G^{i,k+1}_tL 
        - G^{i,k+1}_t\sum_{j\neq i} e_je_j\transpose G^{j,k}_t - \sum_{j\neq i} G^{j,k}_te_je_j\transpose G^{i,k+1}_t = 0,
    \end{multline}
    with terminal condition $G^{i,k+1}_{T'} = \frac{T}{T'}cLe_ie_i\transpose L$.
    This implies that the FP procedure for the linear-quadratic games on graphs~\eqref{eqn:state_dynamics}--\eqref{def:J} with model parameters \((T,a,q,\EPS,c)\) is equivalent to the one with model parameters \((T',a',q',\EPS',c')\), where
    \begin{equation}
        a' = \frac{T}{T'}a,\quad q' = \frac{T}{T'}q,\quad \EPS' = \left(\frac{T}{T'}\right)^2\EPS,\quad c' = \frac{T}{T'}c.
    \end{equation}
    Condition~\eqref{eqn:FP_condition} imposed on the new set of model parameters \((T',a',q',\EPS',c')\) requires:
    \begin{equation}
        \max\left\{[c' + T'(\EPS'-(q')^2)]\;\NORM{L}^2, \frac{2}{\log 2}(a'+q')T'\NORM{L}\right\}\leq 1,\quad T' = \frac{1}{100},
    \end{equation}
    which concludes the proof.
\end{proof}

\begin{rem}[Discussion on condition~\eqref{eqn:FP_final_condition}]
The condition~\eqref{eqn:FP_final_condition} imposes a smallness constraint on the model parameters, a requirement that is commonly found in the literature. 
For instance, the convergence of FP to the open-loop NE for linear-quadratic games on complete graphs requires that \(T\) or \((c, q, \EPS - q^2)\) are sufficiently small, or $a$ is sufficiently large \cite{hu2019deep}. The convergence of the Deep BSDE method \cite{han2017deep} for stochastic control problems with Lipschitz models requires a small time duration or small Lipschitz coefficients \cite{han2022convergence}, as well as the well-posedness of forward-backward stochastic differential equations (FBSDEs) \cite{ma2015well}.
Condition~\eqref{eqn:FP_final_condition} aligns with these requirements,
and it is certainly not the strictest constraint and can be easily extended. As long as reasonable bounds are imposed on model parameters such that the function \(h\) in equation~\eqref{eqn:h_FP} has a fixed point in \(\R_+\), Theorem~\ref{thm:conv_FP} and Corollary~\ref{cor:time_scaling} remain valid and provide a standard way to find a condition, under which FP is guaranteed to converge. 
Lastly, we want to emphasize two points: (i) FP does not necessarily have to start from \(F^{i,0}\equiv 0\). From the proof of Theorem~\ref{thm:conv_FP}, the convergence results still hold as long as \(\sup_{t\in[0,T]}\NORM{\sum_{j=1}^N F^{j,0}_t}\leq x^*\). 
(ii) Condition~\eqref{eqn:FP_final_condition} does not impose any restrictions on the number of players \(N\).
The convergence result still holds even when there is a large number of players.

\end{rem}

\medskip

\begin{rem}[The graph structure and the convergence of fictitious play]

We are particularly interested in the connection between the graph structure and the convergence of fictitious play (FP). This relationship can be analyzed through \(C_L\) defined in condition~\eqref{eqn:FP_final_condition}.
Fixing the other model parameters, an increase in \(\NORM{L}\) leads to an increase in \(C_L\), making it more likely to violate condition~\eqref{eqn:FP_final_condition}.
Taking a closer look, we note that any graph Laplacian \(L\) must satisfy \(0\leq\NORM{L}\leq 2\).
If a graph has no isolated vertices, its graph Laplacian satisfies \(\NORM{L}\geq \frac{N}{N-1}\) \cite{chung1997spectral}.
A connected bipartite graph has the norm of the graph Laplacian achieving the upper bound \(\NORM{L}=2\), while a complete graph achieves the lower bound \(\NORM{L} = \frac{N}{N-1}\).
From the definition of \(g_k(t)\) and the contraction~\eqref{eqn:contrac} in the proof of Theorem~\ref{thm:conv_FP}, we derive an upper bound for the convergence rate of FP:
\begin{equation}
    \sup_{t\in[0,T]}\sup_{i\in[N]}\NORM{\Delta F^{i,k}_t}\leq 2^{-k}e^{6T[(a+q)\NORM{L} + x^* + Nx^*]}\sup_{t\in[0,T]}\sup_{i\in[N]}\NORM{\Delta F^{i,0}_t},\ \forall k\in\N,
\end{equation}
which admits an exponential rate of convergence.
Although the precise convergence rate of FP remains difficult to quantify, numerical experiments show a connection between the graph structure and the convergence rate of FP.
We refer the readers to Section~\ref{sec:num_FP} for a more detailed discussion. 
\label{rem:rate_graph}
\end{rem}

\medskip

Despite the existence of more efficient numerical algorithms for solving the Nash equilibrium in the LQ case, FP provides a novel approach that explicitly constructs a solution to the coupled Riccati system~\eqref{eqn:Riccati} and establishes its existence interval, as shown in Corollary~\ref{cor:wellposed_FP}.

\begin{cor}
    \label{cor:wellposed_FP}
    If condition~\eqref{eqn:FP_final_condition} holds, then the coupled Riccati system~\eqref{eqn:Riccati} admits a unique solution on \([0,T]\).
\end{cor}
\begin{proof}
    Fictitious play provides an explicit construction of the solution, as demonstrated in the proof of Theorem~\ref{thm:conv_FP}. 
    By \cite[Theorem~2.3 and Section~2.4]{lukes1971equilibrium}, the solution to the coupled Riccati system~\eqref{eqn:Riccati} is unique whenever it exists.
    Combining these results concludes the proof.
\end{proof}

To the best of our knowledge, the most general existence interval for the solution to the coupled Riccati system~\eqref{eqn:Riccati} is given in \cite{papavassilopoulos1979existence}, where the range of $T$ explicitly depends on the number of players \(N\). In contrast, our existence condition~\eqref{eqn:FP_final_condition} holds for any \(N\in\N\), even when \(N\) is large.
This highlights the advantages of our approach, which establishes the local well-posedness of coupled Riccati systems through the lens of fictitious play.


\section{Semi-Explicit Equilibrium under Vertex-Transitive Graph}\label{sec:transitive}
In this section, we focus on the game on a vertex-transitive graph \(G\), and aim to provide more explicit characterizations for the value functions $v^i$ and equilibrium strategies $\hat\alpha^i$, compared to the general case characterized by the coupled matrix-valued Riccati equations \eqref{eqn:Riccati}--\eqref{eqn:NE_strategy}. Vertex-transitive graphs, which are defined in Definition~\ref{defn:transitive}, have nice symmetry properties, and play a key role in our analysis. 

\begin{defn}
 A graph \(G\) is vertex-transitive if \(\forall v_i,v_j\in V\), there exists \(\varphi \in \text{Aut}(G)\) such that \(\varphi(v_i) = v_j\). Here \(\text{Aut}(G)\) denotes the set of automorphisms of the graph $G$: 
\begin{equation}
     \text{Aut}(G)\hspace{-2pt}:=\hspace{-2pt}\{\text{bijection } \phi: V\to V \vert (u,v)\in E \text{ if and only if } (\phi(u),\phi(v))\in E, \  \forall u,v\in V \}.
 \end{equation}  \label{defn:transitive}
\end{defn}

In the definition of \(\text{Aut}(G)\), a graph automorphism \(\phi\) should be understood as a bijective relabeling of graph vertices that preserves edge connections.
Algebraically, \(\text{Aut}(G)\) has a group structure under the operation of function composition, which induces a natural group action on the vertex set \(V\) via \(\text{Aut}(G)\times V\ni(\phi,v)\mapsto \phi(v)\in V\).
From an algebraic perspective, a graph is called vertex-transitive if such group action on the vertex set is transitive, i.e., all graph vertices are indistinguishable and structurally equivalent.

The strong symmetry of vertex-transitive graphs facilitates theoretical analysis and makes them appealing for practical applications.
For example, in terms of graph homomorphisms, the graph cores of vertex-transitive graphs have special structures \cite{godsil2001algebraic}.
In terms of spectral properties, tighter spectral bounds can be established for vertex-transitive graphs \cite{chung1997spectral}.
On the application side, vertex-transitive graphs are crucial in communication network designs \cite{tang1992vertex}.

Since a vertex-transitive graph must be a regular graph, Remark~\ref{rem:model} implies that for player \(i\), the mean reversion level in the model dynamics is the arithmetic average of the states of all players in the neighborhood of \(v_i\). As seen below,  the symmetry of the graph allows for the possibility of deriving a semi-explicit form of the Markovian NE. The main result, Theorem~\ref{thm:NE_closed_form}, is presented in Section~\ref{sec:mainresult}, followed by discussions and generalizations. Its constructions and proofs are provided in later sections. To proceed, we first need the following definitions.

Let \(\X\subset \mathbb{S}^{N\times N}\) be a subset of all symmetric $N\times N $ matrices, such that
\begin{equation}
    \X := \SET{X:X \geq 0\text{ is a polynomial of } L},
    \label{eqn:solution_space}
\end{equation}
where the inequality between symmetric matrices is defined in terms of the semi-positive-definite sense. That is,  \(A\geq B\) if and only if \(A-B\) is semi-positive-definite, for \(A,B\in\mathbb{S}^{N\times N}\).

\subsection{Main results}\label{sec:mainresult}

\begin{thm}
\label{thm:NE_closed_form}
Let \(G\) be a simple connected vertex-transitive graph with undirected edges.
Assume \(q^2 = \EPS\) in problem~\eqref{eqn:state_dynamics}--\eqref{def:J}. Then for any \( T>0\), the solution to the following ODE, denoted as \(R:[0,T]\to \X\), exists and is unique.
\begin{equation}
    R'(t) = \frac{1}{c}\Tr \left[Q'(R(t)) e^{-t(a+q)L}\right]e^{-t(a+q)L},\quad R(0) = 0,
    \label{eqn:second_charac}
\end{equation}
where the function \(Q:\X\to \R\) is defined as
\begin{equation}
    Q(X) := \left[\det\left(I+cXL\right)\right]^{\frac{1}{N}}.
    \label{eqn:Q}
\end{equation} 
The solution $F^i$ to the Riccati system~\eqref{eqn:Riccati} is given by:
\begin{equation}
    F^i_t = \frac{1}{\frac{\Tr(P_t) - (a+q)\Tr(L)}{N}}[P_t - (a+q)L]e_ie_i\transpose [P_t - (a+q)L],
    \label{eqn:closed_form_F}
\end{equation}
where  \(P\) is constructed in terms of \(R\) that
\begin{equation}
    P_t = (a+q)L + R'(T-t)cL[I + R(T-t)cL]^{-1}.
    \label{eqn:R_to_P}
\end{equation}
The Markovian NE for player \(i\) is given by
\begin{equation}
    \hat{\alpha}^i(t,x) = -qe_i\transpose L x - e_i\transpose F^i_t x.
    \label{eqn:transitive_NE}
\end{equation}
\end{thm}

\begin{proof}

We first show that $R$ takes values in $\X$. From equation~\eqref{eqn:second_charac}, it is clear that $R$ can be written as
    \begin{equation}
        R(t) = \int_0^t C(s)e^{-s(a+q)L}\,\ud s,
    \end{equation}
where \(C:[0,t]\to\R_+\) is a function containing model parameters. Since \(L\in\mathbb{S}^{N\times N}\), it has a characteristic polynomial \(\Lambda\) with \(\deg\Lambda = N\).
    By the Cayley–Hamilton theorem \cite{lax2007linear}, \(\Lambda(L) = 0\).
    The matrix exponential \(e^{-s(a+q)L} := f_s(L)\) is an analytic function of \(L\), where \(f_s(x) = \sum_{k=0}^\infty \frac{[-s(a+q)]^k}{k!}x^k\).
    There exists polynomials \(M_s\) and \(N_s\), with \(\deg N_s\leq N-1\), such that \(f_s(x) = \Lambda(x)M_s(x) + N_s(x)\). Let $N_s(x)$ be $N_s(x) = \sum_{k=0}^{N-1}a_k(s)x^k$, this results in \(e^{-s(a+q)L} = f_s(L) = N_s(L) = \sum_{k=0}^{N-1} a_k(s) L^k\), \(\forall s\in[0,t]\), which leads to
    \begin{equation}
        R(t) = \sum_{k=0}^{N-1}\left(\int_0^t C(s) a_k(s) \,\ud s\right) L^k.
    \end{equation}
    Hence it is proved that \(R\) must take values in \(\X\). The Markovian NE in equation~\eqref{eqn:transitive_NE} follows directly from equation~\eqref{eqn:NE_strategy}.

Then, it suffices to establish the following: 
(i) The existence and uniqueness of solutions to \eqref{eqn:second_charac}, which will be discussed in Section \ref{subsec:exist_unique}.
(ii) The function \(F^i\) as constructed in equation~\eqref{eqn:closed_form_F} is a solution to the Riccati system~\eqref{eqn:Riccati}. This will be addressed in Section~\ref{subsec:verification}. 
\end{proof}

Equations~\eqref{eqn:second_charac}--\eqref{eqn:R_to_P} provide an explicit construction of a global solution to the Riccati system~\eqref{eqn:Riccati}.
By \cite[Theorem~2.3 and Section~2.4]{lukes1971equilibrium}, the solution to the coupled system~\eqref{eqn:Riccati} is unique whenever it exists.
Combining both results establishes the global well-posedness of system~\eqref{eqn:Riccati} when \(G\) is vertex-transitive and \(\EPS = q^2\). 
This nontrivial result is presented in Corollary~\ref{cor:transitive_Riccati}.

\begin{cor}
    \label{cor:transitive_Riccati}
    Let \(G\) be a simple connected vertex-transitive graph with undirected edges.
    Assume \(q^2 = \EPS\) in problem~\eqref{eqn:state_dynamics}--\eqref{def:J}. Then for any \( T>0\), the coupled Riccati system~\eqref{eqn:Riccati} admits a unique solution on \([0,T]\), given by equations~\eqref{eqn:second_charac}--\eqref{eqn:R_to_P}.
\end{cor}

Compared to Corollary~\ref{cor:wellposed_FP}, which only provides an existence interval for the solution, Corollary~\ref{cor:transitive_Riccati} provides a much stronger result, thanks to the adequate graph symmetry and the uncoupled structure of model parameters (see Remark~\ref{rem:q^2=eps}).

We note that the existence of a global solution to the coupled Riccati system~\eqref{eqn:Riccati} has not yet been fully characterized in the existing literature. While some studies have approached the solvability of coupled Riccati equations from a geometric perspective \cite{papavassilopoulos1984linear}, global existence for any \(T>0\) and \(N\geq 2\) is only guaranteed when all coefficients are diagonal matrices (see \cite[part~(iii) of Remark~6.6.3]{abou2012matrix}), which means that system~\eqref{eqn:Riccati} requires a trivial diagonal graph Laplacian $L$.
By exploiting symmetry and decoupling the equations within the Riccati system~\eqref{eqn:Riccati}, Corollary~\ref{cor:transitive_Riccati} greatly improves existing global well-posedness results.

\medskip

\begin{rem}[The law of the equilibrium state process]
    Plugging the equilibrium strategy~\eqref{eqn:transitive_NE} into the state dynamics~\eqref{eqn:state_dynamics} yields the dynamics of the equilibrium state process \(\hat{X}_t := [\hat{X}_t^1, \ldots, \hat{X}_t^N]\transpose\):
    \begin{equation}
        \,\ud \hat{X}_t = -P_t\hat{X}_t\,\ud t + \sigma\,\ud W_t,
    \end{equation}
    where \(P_t\) takes the form
    $$ P_t = \sum_{k=1}^N \left[(a+q)L + F^k_t\right]e_ke_k\transpose = (a+q)L + \sum_{k=1}^N F^k_te_ke_k\transpose,
    $$ and has the representation~\eqref{eqn:R_to_P} given by Theorem~\ref{thm:NE_closed_form}.
    Given the initial condition \(\hat{X}_0 = X_0\), this Ornstein-Uhlenbeck dynamics has a closed-form solution:
    \begin{equation}
        \hat{X}_t = e^{-\int_0^t P_s\,\ud s}X_0 + \sigma \int_0^t e^{-\int_s^t P_u\,\ud u}\,\ud W_s,
    \end{equation}
    where we have used the fact that \(P_s\) and \(P_t\) commute for any \(t,s\in[0,T]\). 
    Therefore, \(\hat{X}_t\) is Gaussian distributed with mean \(e^{-\int_0^t P_s\,\ud s}X_0\) and variance \(\sigma^2 \int_0^t e^{-2\int_s^t P_u\,\ud u}\,\ud s\).
    Using equations~\eqref{eqn:calc_rho} and~\eqref{eqn:def_R} in Appendix~\ref{sec:heuristic} below, which can be verified to be correct, the mean and variance of \(\hat{X}_t\) can be represented in terms of \(R\), the solution to equation~\eqref{eqn:second_charac}.
\end{rem}

\medskip

We note the following, (i) Theorem~\ref{thm:NE_closed_form} applies specifically to the case where \(q^2 = \EPS\).
    As will discussed in Remark~\ref{rem:q^2=eps}, this condition results in an uncoupled nature of the underlying linear system, greatly simplifying the calculations.
(ii) A numerical verification of Theorem~\ref{thm:NE_closed_form} will be presented in Section~\ref{sec:num_veri}.
(iii) Theorem~\ref{thm:NE_closed_form} can be generalized to provide a semi-explicit Markovian NE for games where the mean reversion level is the average of states within a more broadly defined neighborhood of a player, such as the $l$-neighborhood (see Section~\ref{sec:l_nbhd}). This extends the concept of a graph neighborhood as introduced in  Definition~\ref{defn:graph}. The generalized result is briefly outlined below.

\subsubsection{Generalizations to $l-$neighborhood interactions}\label{sec:l_nbhd}
\begin{defn}
    The \(l\)-neighborhood of a vertex \(v\) in a graph \(G\), denoted \(N_G^l(v)\), is the collection of all vertices that are exactly \(l\) steps away from \(v\), i.e.,
    \begin{equation}
        N_G^l(v) = \SET{u\in V: \text{there exists a path of length} \ l\ \text{between}\ u\ \text{and}\ v}. 
    \end{equation} 
    \label{defn:l_nbhd}
\end{defn}

To set the mean reversion level as the average within the $l$-neighborhood, we consider the following generalized state dynamics on a vertex-transitive graph \(G\):
\begin{equation}
    \ud X^i_t = \left[a\left(\sum_{j:v_j\in N_G^l(v_i)}\frac{n(v_j,v_i;l)}{|N_G^l(v_i)|}X^j_t - X^i_t\right) + \alpha^i_t\right]\,\ud t + \sigma\,\ud W^i_t, \quad \forall i\in[N],
    \label{eqn:generailzed_state_dyn}
\end{equation}
where \(n(v_j,v_i;l)\) stands for the number of paths of length \(l\) between \(v_i\) and \(v_j\).
Similarly, the running and terminal costs of player \(i\in[N]\) are set as
\begin{align}    \label{eqn:generalized_running_cost}
    f^i(t,x,\alpha) &= \frac{1}{2}(\alpha)^2 - q\alpha \Big(\sum_{j:v_j\in N_G^l(v_i)}\frac{n(v_j,v_i;l)}{|N_G^l(v_i)|}x^j - x^i\Big) \\
    & \qquad + \frac{\EPS}{2}\Big(\sum_{j:v_j\in N_G^l(v_i)}\frac{n(v_j,v_i;l)}{|N_G^l(v_i)|}x^j - x^i\Big)^2,\\
    g^i(x) &= \frac{c}{2}\Big(\sum_{j:v_j\in N_G^l(v_i)}\frac{n(v_j,v_i;l)}{|N_G^l(v_i)|}x^j - x^i\Big)^2.
    \label{eqn:generalized_terminal_cost}
\end{align}
In the generalized state dynamics \eqref{eqn:generailzed_state_dyn}, the mean reversion level is interpreted as a weighted average of all the players' states within \(N^l_G(v_i)\).
Each vertex \(v_j\in N^l_G(v_i)\) is assigned a weight of \(n(v_j,v_i;l)/|N_G^l(v_i)|\), which is proportional to the number of paths of length \(l\) between \(v_i\) and \(v_j\).

To express the mean reversion level in a matrix form, as done in equation~\eqref{eqn:mean_rev_graph_lap},
we define \(A\) as the adjacency matrix of \(G\), and \(\delta\) as the common degree shared by all vertices in \(G\).
Since the graph is vertex-transitive, the cardinality of the \(l\)-neighborhood can be  calculated as follows: 
\begin{equation}
    |N^l_G(v_i)| = \sum_{j=1}^N (A^l)_{ij} = \sum_{k=1}^N(A^{l-1})_{ik}\sum_{j=1}^NA_{kj} = \delta\sum_{k=1}^N(A^{l-1})_{ik} = \cdots = \delta^l,
\end{equation}
for any \(v_i\in V\).
Due to the properties of the adjacency matrix, \(n(v_j,v_i;l) = e_i\transpose A^l e_j\), allowing us to rewrite the mean reversion level as \(\sum_{j:v_j\in N_G^l(v_i)}\frac{n(v_j,v_i;l)}{|N_G^l(v_i)|}X^j_t = e_i\transpose (A^l/\delta^l) X_t\).
According to Definition~\ref{defn:graph_lap},  \(L = I-\frac{1}{\delta}A\). Thus,
\begin{equation}
    \sum_{j:v_j\in N_G^l(v_i)}\frac{n(v_j,v_i;l)}{|N_G^l(v_i)|}X^j_t - X^i_t = -e_i\transpose M(l,L)X_t, \quad \text{ where}\quad  M(l,L) = I - (I-L)^l.
\end{equation}

For the game described above, under the assumption that \(q^2 = \EPS\), the Markovian NE can also be constructed semi-explicitly, similar to Theorem~\ref{thm:NE_closed_form}. The only modification needed is replacing the graph Laplacian \(L\) in Theorem~\ref{thm:NE_closed_form} with \(M(l,L)\).

Another game with favorable results occurs when the mean reversion level is a weighted average of players' states that are at most $l$ steps away, i.e., averaging among  \(v_j \in \bigcup_{k=1}^l N^k_G(v_i)\). In this case, the mean reversion term in the state dynamics and cost functions is:
\begin{equation}\label{eqn:generailzed_state_dyn2}
    \sum_{j:v_j\in \bigcup_{k=1}^lN_G^k(v_i)}\frac{\sum_{k=1}^l w^k n(v_j,v_i;k)}{\sum_{k=1}^lw^k|N_G^k(v_i)|}X^j_t - X^i_t, 
\end{equation}
where each vertex \(v_j\in \bigcup_{k=1}^l N^k_G(v_i)\) is assigned a weight proportional to a discounted sum of the number of paths of length \(k\in\SET{1,\ldots,l}\) between \(v_i\) and \(v_j\).
Each path of length \(k\) is discounted by a factor of \(w^k\), where $w$ denotes the discount rate.

The expression in equation~\eqref{eqn:generailzed_state_dyn2} can be written in matrix form $-e_i\transpose S(l,L)X_t$, where
\begin{equation}
  S(l,L) = I -  \frac{1-w\delta}{1-(w\delta)^l}(I-L)[I-(w\delta)^l(I-L)^l][I-w\delta(I-L)]^{-1}.
\end{equation}
Then, under the assumption that \(q^2 = \EPS\) and \(w\in(0,1/\delta)\), a semi-explicit Makovian NE can be constructed similarly to  Theorem~\ref{thm:NE_closed_form}, with the graph Laplacian $L$ replaced by $S(l, L)$. 
    \label{rem:l_nbhd}

\subsubsection{Proof roadmap}\label{sec:roadmap}
This subsection briefly outlines the key steps leading to formulation \eqref{eqn:closed_form_F}, with further details provided in Appendix~\ref{sec:heuristic}.

The construction of  \eqref{eqn:closed_form_F}, together with \eqref{eqn:second_charac}--\eqref{eqn:Q} and \eqref{eqn:R_to_P}, is inspired by a previous work \cite{lacker2022case} and employs two key techniques: a fixed point scheme that helps represent the solution to the Riccati system, and the regular representation of the graph automorphism group that leverages symmetry. 

\smallskip
\noindent{\textbf{Step 1:}}  We first define the quantity  
$$
 P_t := \sum_{k=1}^N \left[(a+q)L + F^k_t\right]e_ke_k\transpose = (a+q)L + \sum_{k=1}^N F^k_te_ke_k\transpose,
$$
which appears in system~\eqref{eqn:Riccati}. Solving for $F^i$ in terms of $P$ yields \eqref{eqn:solution_Riccati_P}.

\smallskip
\noindent{\textbf{Step 2:}} Next, we substitute  equation~\eqref{eqn:solution_Riccati_P} into the above definition of $P$, introduce a new quantity $\eta$ (see equation~\eqref{eqn:eta}), and express $P$ in terms of $\eta$, leading to equation~\eqref{eqn:rep_P_use_eta}.

\smallskip
\noindent{\textbf{Step 3}:} Using the definition of $\eta$,  we derive the governing equation:
$$
R'(T-t) = \eta_{t}e^{-(T-t)(a+q)L},
$$
where $R$ solves equation~\eqref{eqn:second_charac}.

\smallskip
\noindent{\textbf{Step 4}:} 
Once $R$ is obtained from the matrix-valued ODE~\eqref{eqn:second_charac}, and $\eta$ is constructed, we can represent $P_t$ in terms of $R$, rather than in terms of $\eta$, leading to \eqref{eqn:R_to_P}. Finally, combining equations~\eqref{eqn:R_to_P} and \eqref{eqn:solution_Riccati_P} gives \eqref{eqn:closed_form_F}. 

The full details can be found in Appendices~\ref{app:s1}--\ref{sec:close_the_loop}. The following two subsections take \eqref{eqn:closed_form_F} as given and focus on establishing the well-posedness of equation~\ref{eqn:second_charac}, which ensures the existence of $\eta$, and on verifying that equation~\eqref{eqn:closed_form_F} indeed solves the Riccati system~\eqref{eqn:Riccati}.

\subsection{The well-posedness of equation~\texorpdfstring{\eqref{eqn:second_charac}}{} and the existence of \texorpdfstring{$\eta$}{}}
\label{subsec:exist_unique}

Due to the essential dependence of the semi-explicit Markovian NE on $R(t)$, it is both natural and essential to 
establish the existence and uniqueness of solutions to~\eqref{eqn:second_charac}.
We recall readers of equations~\eqref{eqn:solution_space}--\eqref{eqn:Q} as the definitions of \(\X\), \(Q\), and the matrix-valued ODE for \(R(t)\):
\begin{align}
&\X := \SET{X:X \geq 0\text{ is a polynomial in } L}, \tag{\ref{eqn:solution_space}}\label{eqn:solution_space_recall}\\
& R'(t) = \frac{1}{c}\Tr \left[Q'(R(t)) e^{-t(a+q)L}\right]e^{-t(a+q)L},\quad R(0) = 0. \tag{\ref{eqn:second_charac}}\label{eqn:second_charac_recall}\\
& Q(X) := \left[\det\left(I+cXL\right)\right]^{\frac{1}{N}}, \tag{\ref{eqn:Q}}\label{eqn:Q_recall}
\end{align}
The definition of \(\X\) in~\eqref{eqn:solution_space_recall} is proposed such that any two matrices in the set \(\SET{L,X_1,X_2}\) commute, allowing them to be simultaneously diagonalizable for any \(X_1,X_2\in\X\). Next, we establish the well-posedness of equation~\eqref{eqn:second_charac_recall} within the set $\X$. We begin by proving a local result and then extend it using a concavity argument.

\begin{thm}
There exists \(t_0\in\R_+\) and \(x_0\in\R_+\) such that the solution \(R(t)\) to equation~\eqref{eqn:second_charac}
exists and is unique in \(\X\), when \(t\in [0,t_0]\) and \(0\leq R(t)\leq x_0I\).
\label{thm:local_exist_unique}
\end{thm}

\begin{proof}

Define the function \(f:[0,T]\times \X\to \X\) such that
\begin{equation}
    f(t,X) = \frac{1}{c}\Tr \left[Q'(X) e^{-t(a+q)L}\right]e^{-t(a+q)L},
    \label{eqn:f_ODE_proof}
\end{equation}
with $Q(X)$ defined in \eqref{eqn:Q_recall}. 
Then equation~\eqref{eqn:second_charac} can be rewritten as \(R'(t) = f(t,R(t))\) with the initial condition \(R(0) = 0\). In the following, we use the Picard-Lindelöf theorem \cite{teschl2012ordinary} to prove the local existence and uniqueness of the solution.

We first check that  \(\X\) is closed under the Picard iteration. Given the $k^{th}$ function $R_k(t)$, the $(k+1)^{th}$ iteration is given by:
\begin{equation}
    R_{k+1}(t) = \int_0^t \frac{1}{c}\Tr \left[Q'(R_k(s)) e^{-s(a+q)L}\right]e^{-s(a+q)L}\,\ud s.
\end{equation}
Following the argument in the first part of the proof of Theorem~\ref{thm:NE_closed_form}, if \(R_k\) takes values in \(\X\), then \(R_{k+1}\) also takes values in \(\X\).

We proceed to prove that \(f\) is continuous in \(t\) and Lipschitz continuous in \(X\) for all \(t\in [0, t_0]\) and \(X\in\X\) such that \(X\leq x_0 I\).
Since the continuity in \(t\) is straightforward, our focus will be on proving the Lipschitz continuity in \(X\).
Given that all eigenvalues of \(L\) lie in the interval \([0,2]\), it's clear that \(\NORM{e^{-t(a+q)L}}\leq 1\).
Combining it with the Cauchy-Schwarz inequality under the matrix inner product \(\langle A,B\rangle := \Tr(A\transpose B)\) yields
\begin{equation}
    \NORM{f(t,X_2)-f(t,X_1)}\leq\frac{1}{c} \sqrt{\Tr\left(e^{-2t(a+q)L}\right)}\sqrt{\Tr\left[\left(Q'(X_2)-Q'(X_1)\right)^2\right]}.
    \label{eqn:bound_two_trace}
\end{equation}
The first term on the right-hand side of the above inequality has a trivial bound:
\begin{equation}
    \sqrt{\Tr\left(e^{-2t(a+q)L}\right)} 
    = \sqrt{\sum_{k=1}^N e^{-2t(a+q)\lambda_k}}\leq \sqrt{N}.
    \label{eqn:bound_first_trace}
\end{equation}
Therefore, we focus on bounding the second term on the right-hand side of inequality~\eqref{eqn:bound_two_trace}.

Denote the eigenvalues of \(X\) as \(x_1,\cdots,x_N\), it's clear that  $0\leq x_i\leq x_0$, for any $i\in[N]$.
Since \(L\) and \(\forall X\in\X\) are simultaneously diagonalizable, an upper bound for \(Q(X)\) can be established:
\begin{equation}
    Q(X) = \left[\prod_{k=1}^N(1+cx_k\lambda_k)\right]^{\frac{1}{N}}\leq 1+2cx_0.
    \label{eqn:UB_for_Q}
\end{equation}
Using Corollary~\ref{cor:matrix_calc} for \(Q(X)\) and combining it with the triangle inequality gives
\begin{multline}
   \hspace{-7pt} \Tr\left[\left(Q'(X_2)-Q'(X_1)\right)^2\right]
    \leq \frac{2c^2}{N^2}\Tr\left(L^2\left[Q(X_2)(I+cX_2L)^{-1} \hspace{-7pt} - Q(X_2)(I+cX_1L)^{-1}\right]^2\right) \\
    +\frac{2c^2}{N^2}\Tr\left(L^2\left[Q(X_2)(I+cX_1L)^{-1}- Q(X_1)(I+cX_1L)^{-1}\right]^2\right).
    \label{eqn:calc_triangle}
\end{multline}
Now, it remains to bound the two trace terms on the right-hand side of inequality~\eqref{eqn:calc_triangle}.

Denote by \(x_{j,1},\cdots,x_{j,N}\) the eigenvalues of \(X_j\) where \(j\in\{1,2\}\).
Notice that any two matrices in the set \(\{X_1,X_2,L\}\) are simultaneously diagonalizable for any \(X_1,X_2\in\X\).
 This yields
\begin{multline}
    \Tr\left[\left(Q'(X_2)-Q'(X_1)\right)^2\right]
    \leq \frac{8c^2}{N^2}\sum_{k=1}^N\left[Q(X_2)\frac{1}{1+cx_{2,k}\lambda_k} - Q(X_2)\frac{1}{1+cx_{1,k}\lambda_k}\right]^2 \\
    + \frac{8c^2}{N^2}\sum_{k=1}^N\left[Q(X_2)\frac{1}{1+cx_{1,k}\lambda_k}- Q(X_1)\frac{1}{1+cx_{1,k}\lambda_k}\right]^2.
    \label{eqn:two_sum_spectrum}
\end{multline}
We next bound the two summations on the right-hand side of inequality~\eqref{eqn:two_sum_spectrum}, respectively.

For the first summation in inequality~\eqref{eqn:two_sum_spectrum}, using equation~\eqref{eqn:UB_for_Q} together with the facts that \(1+cx_{1,k}\lambda_k\geq \frac{1}{2}\) and \(\NORM{X_2-X_1} = \max_{k\in[N]}|x_{2,k}-x_{1,k}|\) produces
\begin{align}
    \sum_{k=1}^N &\left[Q(X_2)\frac{1}{1+cx_{2,k}\lambda_k} - Q(X_2)\frac{1}{1+cx_{1,k}\lambda_k}\right]^2 \\
    &\leq \sum_{k=1}^N (1+2cx_0)^2 \left[\frac{c\lambda_k(x_{1,k}-x_{2,k})}{(1+cx_{1,k}\lambda_k)(1+cx_{2,k}\lambda_k)}\right]^2\\
    &\leq 64c^2(1+2cx_0)^2N\NORM{X_2-X_1}^2.
    \label{eqn_first_sum_bound}
\end{align}

For the second summation in inequality~\eqref{eqn:two_sum_spectrum}, using the bound for \(1+cx_{1,k}\lambda_k\) once again yields
\begin{equation}
    \sum_{k=1}^N\left[Q(X_2)\frac{1}{1+cx_{1,k}\lambda_k}- Q(X_1)\frac{1}{1+cx_{1,k}\lambda_k}\right]^2
    \leq 4\sum_{k=1}^N [Q(X_2)-Q(X_1)]^2.
    \label{eqn:eqn_second_term_RHS}
\end{equation}
Applying the intermediate value theorem for each entry of \(Q'(X)\), one obtains the bound:
\begin{equation}
    [Q(X_2)-Q(X_1)]^2\leq \NORM{X_2-X_1}_F^2\ \sup_{\tiny{\mathclap{\substack{0\leq X\leq x_0I\\ X\in\X}}}}\NORM{Q'(X)}_F^2
    \leq N^2\NORM{X_2-X_1}^2\ \sup_{\tiny\mathclap{\substack{0\leq X\leq x_0I\\ X\in\X}}}\NORM{Q'(X)}^2,
    \label{eqn:eqn_with_sup}
\end{equation}
where \(\NORM{\cdot}_F\) denotes the Frobenius norm.
The supremum term in inequality~\eqref{eqn:eqn_with_sup} is bounded by:
\begin{equation}
    \sup_{\tiny\mathclap{\substack{0\leq X\leq x_0I\\ X\in\X}}}\NORM{Q'(X)}^2
    \leq \frac{c}{N}(1+2cx_0)\sup_{\tiny0\leq x_1,\cdots,x_n\leq x_0}\sup_{\tiny k\in [N]}\frac{\lambda_k}{1+cx_k\lambda_k}
    \leq \frac{4c}{N}(1+2cx_0).
    \label{eqn:bound_sup}
\end{equation}
Plugging inequalities~\eqref{eqn:eqn_with_sup}--\eqref{eqn:bound_sup} into inequality~\eqref{eqn:eqn_second_term_RHS} gives
\begin{equation}
    \sum_{k=1}^N\left[Q(X_2)\frac{1}{1+cx_{1,k}\lambda_k}- Q(X_1)\frac{1}{1+cx_{1,k}\lambda_k}\right]^2
    \leq 64c^2N(1+2cx_0)^2\NORM{X_2-X_1}^2.
    \label{eqn:second_term_bound}
\end{equation}
Combining inequalities~\eqref{eqn:bound_two_trace}--\eqref{eqn:bound_first_trace},~\eqref{eqn:two_sum_spectrum}--\eqref{eqn_first_sum_bound} and~\eqref{eqn:second_term_bound} yields
\begin{equation}
    \NORM{f(t,X_2)-f(t,X_1)}
    \leq 32c(1+2cx_0)\NORM{X_2-X_1}.
    \label{eqn:f_Lips}
\end{equation}
Therefore, \(f\) is Lipschitz continuous in \(X\) with Lipschitz constant \(32c(1+2cx_0)\). 
This concludes the proof.
\end{proof}

In general, without the local constraint \(R(t)\leq x_0I\), the function \(f\) defined in \eqref{eqn:f_ODE_proof} fails to be globally Lipschitz.
However, through a concavity argument, we prove that the solution \(R(t)\) has an \textit{a priori} upper bound. This upper bound implies the global existence and uniqueness of the solution, summarized as follows.

\begin{thm}
    For any \(T>0\), the solution \(R(t)\) to equation~\eqref{eqn:second_charac} exists and is unique in \(\X\) on the finite time horizon \(t\in[0,T]\).
\label{thm:global_exist_unique}
\end{thm}

\begin{proof}
    Notice that the proof of Theorem~\ref{thm:local_exist_unique} does not actually rely on \(t_0\).
    Therefore, it suffices to establish an a priori upper bound for \(R(t)\), i.e., finding \(x_0>0\) such that $R(t)\leq x_0I$, for all $t\in[0,T]$.

   Let \(r_1(t),\cdots,r_N(t)\) denote the eigenvalues of \(R(t)\), such that \(r_1'(t),\cdots,r_N'(t)\) represent the eigenvalues of \(R'(t)\).  It's clear that \(r_1(t),\cdots,r_N(t)\) are positive for \(\forall t\in[0,T]\). Since any two matrices in the set \(\{L,R(t),R'(t)\}\) commute for any fixed \(t\in[0,T]\), they can be simultaneously diagonalized. Rewriting equation~\eqref{eqn:second_charac}  using the spectra of matrices, one has
    \begin{equation}
            r_k'(t) = \frac{1}{N}\left[\prod_{j=1}^N(1+c\lambda_jr_j(t))\right]^{\frac{1}{N}}\sum_{j=1}^N \frac{\lambda_je^{-t(a+q)\lambda_j}}{1+c\lambda_jr_j(t)}\cdot e^{-t(a+q)\lambda_k},\quad r_k(0) = 0.
        \label{eqn:second_charac_spectrum}
    \end{equation}
    Define \(y_j(t) := \frac{\lambda_je^{-t(a+q)\lambda_j}}{1+c\lambda_jr_j(t)}\), and \(Q_R(t) := Q(R(t)) = \left[\prod_{j=1}^N(1+c\lambda_jr_j(t))\right]^{\frac{1}{N}}\) to capture the dependence of \(Q(R(t))\) on time \(t\). 
    Taking the logarithm of \(Q_R(t)\) and then differentiating with respect to \(t\) yields:
    \begin{equation}
        \frac{\ud}{\ud t}Q_R(t)
        = Q_R(t) \frac{1}{N} \sum_{j=1}^N \frac{c\lambda_jr_j'(t)}{1+c\lambda_jr_j(t)}.
        \label{eqn:Q_R_derivative_tmp}
    \end{equation}
    Plugging equation~\eqref{eqn:second_charac_spectrum} into the equation above produces
    \begin{equation}
        \frac{\ud}{\ud t}Q_R(t)
        = \frac{Q_R^2(t)}{N^2}\sum_{k=1}^N \frac{\lambda_k e^{-t(a+q)\lambda_k}}{1+c\lambda_kr_k(t)}\sum_{j=1}^N \frac{c\lambda_j e^{-t(a+q)\lambda_j}}{1+c\lambda_jr_j(t)} = \frac{c}{N^2}Q_R^2(t)\left(\sum_{j=1}^N y_j(t)\right)^2.
        \label{eqn:Q_R_derivative}
    \end{equation}
    
    In what follows, we establish the concavity of $Q_R$, which is crucial for achieving global well-posedness. 
    Differentiating equation~\eqref{eqn:Q_R_derivative} with respect to $t$ yields:
    \begin{equation}
        \frac{\ud^2}{\ud t^2}Q_R(t)
        = \frac{2c}{N^2}Q_R(t)\cdot\sum_{j=1}^N y_j(t)\cdot \left[\frac{\ud}{\ud t}Q_R(t)\cdot\sum_{j=1}^N y_j(t) + Q_R(t)\sum_{j=1}^N \frac{\ud}{\ud t}y_j(t)\right].
        \label{eqn:calc_Q_second_derivative}
    \end{equation}
    Using equation~\eqref{eqn:second_charac_spectrum}, the derivative of \(y_j(t)\) is computed as follows:
    \begin{equation}
        \frac{\ud}{\ud t}y_j(t) = -(a+q)\lambda_jy_j(t) - cQ_R(t)y_j^2(t)\frac{1}{N}\sum_{k=1}^N y_k(t).
        \label{eqn:y_derivative}
    \end{equation}
    Plugging equations~\eqref{eqn:Q_R_derivative} and~\eqref{eqn:y_derivative} into equation~\eqref{eqn:calc_Q_second_derivative}, we compute the second-order derivative:
    \begin{multline}
        \frac{\ud^2}{\ud t^2}Q_R(t) 
        = \frac{2c}{N^2}Q_R(t)\sum_{j=1}^N y_j(t)\Biggr[\frac{c}{N^2}Q_R^2(t)\left(\sum_{j=1}^N y_j(t)\right)^3 \\
        - \frac{cQ_R^2(t)}{N}\sum_{j=1}^N y_j^2(t)\sum_{j=1}^N y_j(t) -(a+q)Q_R(t)\sum_{j=1}^N \lambda_jy_j(t)\Biggr].
        \label{eqn:Q_R_second_derivative}
    \end{multline}
    Using the trivial bound \((a+q)Q_R(t)\sum_{j=1}^N \lambda_jy_j(t)\geq 0\) and the Cauchy-Schwarz inequality yields
{\small    \begin{equation}
        \frac{\ud^2}{\ud t^2}Q_R(t)
        \leq \frac{2c^2}{N}Q_R^3(t)\sum_{j=1}^N y_j(t)\left[\left(\frac{1}{N}\sum_{j=1}^N y_j(t)\right)^3 - \left(\frac{1}{N}\sum_{j=1}^N y_j(t)\right)\left(\frac{1}{N}\sum_{j=1}^N y_j^2(t)\right)\right] \leq 0.
        \label{eqn:Q_concave}
    \end{equation}} This proves the concavity of \(Q_R\) in \(t\).

    The concavity of $Q_R$ yields the following estimate:
    \begin{equation}
        Q_R(t) \leq Q_R(0) + Q_R'(0)t,\ \forall T>0,\forall t\in[0,T].
    \end{equation}
    By evaluating $Q_R'(0)$ using equation~\eqref{eqn:Q_R_derivative}, one has $Q_R(t) \leq 1+ct$.
    Since \(y_j(t)\leq \lambda_j\), the inequality, combined with equation~\eqref{eqn:second_charac_spectrum}, implies \(r_k'(t)\leq 1+ct\).
    Integrating both sides w.r.t. \(t\), we obtain an a priori bound \(
        r_k(t) \leq \int_0^t (1+cs)\,\ud s \leq T + \frac{c}{2}T^2\).
    Therefore, setting \(x_0 = T + \frac{c}{2}T^2\) completes the proof.
\end{proof} 

By proving Theorem~\ref{thm:global_exist_unique}, we have established the first part of Theorem~\ref{thm:NE_closed_form}.

\subsection{Verification}\label{subsec:verification}

This section aims to verify that equation~\eqref{eqn:closed_form_F} is indeed the solution to the Riccati system~\eqref{eqn:Riccati}. Although proving the existence and uniqueness of the solution $R(t)$ to the governing equation~\eqref{eqn:second_charac} is convenient, using \eqref{eqn:second_charac} for verification complicates the process. 

To achieve this, we propose an alternative governing system for $\eta$, showing its connection with equation~\eqref{eqn:second_charac}, which allows for an easier verification procedure. We once again emphasize that verifying the assumptions in Remark~\ref{rem:assumptions} is unnecessary since the fact that $F_t^i$ given in \eqref{eqn:closed_form_F} solves \eqref{eqn:Riccati} is directly verified below.

\subsubsection{The alternative governing equation for \texorpdfstring{\(\eta\)}{}}

Consider the coupled ODE system: 
\begin{equation}
    \begin{cases}
        Q_1(t) = \left[\det(I+cLR_1(t))\right]^{\frac{1}{N}}\\
        R_1'(T-t) = \sqrt{\frac{Q_1'(T-t)}{c}}e^{-(T-t)(a+q)L}
    \end{cases},\quad  R_1(0) = 0, \quad R_1'(0) = I,
    \label{eqn:first_charac}
\end{equation}
where the solutions \(Q_1:[0,T]\to \R\) and \(R_1:[0,T]\to \X\) are defined on $[0,T]$. We want to relate the above system to equation~\eqref{eqn:second_charac}. Specifically, we demostrate that any solution \(R\) to equation~\eqref{eqn:second_charac} must also provide a solution \((Q_1,R_1 \equiv R)\) to equation~\eqref{eqn:first_charac}, thus equation~\eqref{eqn:first_charac} offers an alternative characterization for $\eta_t$.

Let \(R(t)\) be a solution to equation~\eqref{eqn:second_charac}, we first check the initial conditions in equation~\eqref{eqn:first_charac}.
By Corollary~\ref{cor:matrix_calc}, \(Q'(0) = \frac{c}{N}L\), thus \(R'(0) = \frac{1}{c}\Tr \left[Q'(0) \right]I = I\).
So it remains to show
\begin{equation}
    \left(\Tr \left[Q'(R(t)) e^{-t(a+q)L}\right]\right)^2  = cQ_1'(t),
    \label{eqn:connection}
\end{equation}
with \(Q_1\) defined in \eqref{eqn:first_charac}.
By Corollary~\ref{cor:matrix_calc} and equation~\eqref{eqn:Q}, one has
\begin{equation}
    \Tr \left(Q'(R(t)) e^{-t(a+q)L}\right)
    = \Tr \left(\frac{1}{N}\left[\det\left(I+cLR(t)\right)\right]^{\frac{1}{N}}(I+cLR(t))^{-1} cL e^{-t(a+q)L}\right).
    \label{eqn:trace_equation}
\end{equation}
On the other hand, a calculation of \(Q_1'(t)\) based on Corollary~\ref{cor:matrix_calc} provides
\begin{equation}
    Q_1'(t) = \Tr \left(\frac{1}{N}\left[\det\left(I+cLR(t)\right)\right]^{\frac{1}{N}}(I+cLR(t))^{-1} cL R'(t)\right).
\end{equation}
Replacing $R'(t)$ in the above equation using equation~\eqref{eqn:second_charac} yields
\begin{multline}
     Q_1'(t) = \frac{1}{c}\Tr \left(Q'(R(t)) e^{-t(a+q)L}\right)\\\Tr \left(\frac{1}{N}\left[\det\left(I+cLR(t)\right)\right]^{\frac{1}{N}}(I+cLR(t))^{-1} cL e^{-t(a+q)L}\right),
\end{multline}
which achieves equation~\eqref{eqn:connection}.

\subsubsection{Verification through the alternative governing equation for \texorpdfstring{\(\eta\)}{}}

To conclude the proof of Theorem~\ref{thm:NE_closed_form}, it remains to prove that  \(F^i_t\) provided by equation~\eqref{eqn:closed_form_F} is a solution to the Riccati system~\eqref{eqn:Riccati}. We first check the terminal condition of the Riccati system~\eqref{eqn:Riccati}.
Since \(R(0) = 0\) and \(R'(0) = I\), equation~\eqref{eqn:R_to_P} implies that \(P_T = (a+q)L+cL\).
Combining this with equation~\eqref{eqn:closed_form_F} yields \(F^i_T = cLe_ie_i\transpose L\). To verify that equation~\eqref{eqn:closed_form_F} follows the Riccati system~\eqref{eqn:Riccati}, we need to account for $P_t$, which depends on $\eta_t$. This is where we will use the alternative governing system~\eqref{eqn:first_charac}. 
The verification process is divided into several steps.

\medskip

\noindent\textbf{Step 1: Rebuild the relationship between \(F^i_t\) and \(P_t\).}
Let us define $\tau(t)$ as follows:
\begin{equation}
    \tau(t) := \frac{\Tr(P_t) - (a+q)\Tr(L)}{N}.
    \label{eqn:tau}
\end{equation}
Based on equation~\eqref{eqn:R_to_P} and Lemma~\ref{lemma:group_rep}(i), \(P_t - (a+q)L\) 
takes values as symmetric matrices and commutes with \(R_\varphi\) for all \( \varphi\in\text{Aut}(G)\) for any \(t\in[0,T]\). Lemma~\ref{lemma:group_rep}(ii) implies that \(e_k\transpose [P_t - (a+q)L]e_k = \tau(t)\) for each \(k\in[N]\). Combining this with equation~\eqref{eqn:closed_form_F} yields
\begin{equation}
    \sum_{k=1}^N F^k_te_ke_k\transpose + (a+q)L 
    = \sum_{k=1}^N [P_t-(a+q)L]e_ke_k\transpose + (a+q)L = P_t.
    \label{eqn:validity_P}
\end{equation}
This affirms the validity of equation~\eqref{eqn:defn_P} within the context of verification, which has been a key element throughout the heuristic construction of the equilibrium.

\medskip

\noindent\textbf{Step 2: The spectral form of the Riccati system~\eqref{eqn:Riccati}.} Using the definition of $\tau(t)$ from \eqref{eqn:tau}, equation~\eqref{eqn:closed_form_F} can be rewritten as \(F^i_t = \frac{1}{\tau(t)}[P_t - (a+q)L]e_ie_i\transpose [P_t - (a+q)L]\).
With this expression regarding the Riccati system~\eqref{eqn:Riccati}, it suffices to verify:
\begin{multline}
    -\frac{\tau'(t)}{\tau(t)}[P_t-(a+q)L]e_ie_i\transpose[P_t-(a+q)L] + \dot{P}_te_ie_i^T[P_t-(a+q)L] + [P_t-(a+q)L]e_ie_i\transpose\dot{P}_t\\
    -P_t[P_t-(a+q)L]e_ie_i\transpose[P_t-(a+q)L] - [P_t-(a+q)L]e_ie_i\transpose[P_t-(a+q)L]P_t\\
    + \tau(t)[P_t-(a+q)L]e_ie_i\transpose[P_t-(a+q)L] = 0.
    \label{eqn:check_eqn_tau}
\end{multline}

Since any two matrices in the set \(\{P_t,L,R(T-t),R'(T-t)\}\) commute for any fixed \(t\in[0,T]\), they are simultaneously diagonalizable. Recall that we denote $\lambda_1, \ldots, \lambda_N$ as the eigenvalues of $L$, and $\rho^1_t, \ldots, \rho^N_t$ as the eigenvalues of $P_t$. According to equation~\eqref{eqn:first_charac},  the \(j\)-th eigenvalues of \(R(T-t)\) and \(R'(T-t)\) are \(\int_t^T\sqrt{\frac{Q_1'(T-s)}{c}}e^{-(T-s)(a+q)\lambda_j}\,\ud s\) and \(\sqrt{\frac{Q_1'(T-t)}{c}}e^{-(T-t)(a+q)\lambda_j}\) respectively.
As a result, equation~\eqref{eqn:R_to_P} can be expressed in the spectral form as
\begin{equation}
    \rho^j_t = (a+q)\lambda_j + \frac{c\lambda_j\sqrt{\frac{Q_1'(T-t)}{c}}e^{-(T-t)(a+q)\lambda_j}}{1 + c\lambda_j\int_t^T\sqrt{\frac{Q_1'(T-s)}{c}}e^{-(T-s)(a+q)\lambda_j}\,\ud s}.
    \label{eqn:rho_veri}
\end{equation}
To verify equation~\eqref{eqn:check_eqn_tau}, it suffices to verify its spectral version:
\begin{multline}
    -\frac{\tau'(t)}{\tau(t)}[\rho^j_t-(a+q)\lambda_j][\rho^k_t-(a+q)\lambda_k] + \dot{\rho}^j_t[\rho^k_t-(a+q)\lambda_k] + [\rho^j_t-(a+q)\lambda_j]\dot{\rho}^k_t\\
    -\rho^j_t[\rho^j_t-(a+q)\lambda_j][\rho^k_t-(a+q)\lambda_k] - [\rho^j_t-(a+q)\lambda_j][\rho^k_t-(a+q)\lambda_k]\rho^k_t\\
    + \tau(t)[\rho^j_t-(a+q)\lambda_j][\rho^k_t-(a+q)\lambda_k] = 0,\ \forall j,k\in[N],
    \label{eqn:check_eqn_tau_spectral}
\end{multline}
whose left-hand side can be factored as the sum of two products:
\begin{multline}
    \left(\dot{\rho}^k_t - [\rho^k_t-(a+q)\lambda_k]\rho^k_t - \frac{\rho^k_t-(a+q)\lambda_k}{2}\left[\frac{\tau'(t)}{\tau(t)} - \tau(t)\right]\right)[\rho^j_t-(a+q)\lambda_j]\\
    +\left(\dot{\rho}^j_t - [\rho^j_t-(a+q)\lambda_j]\rho^j_t- \frac{\rho^j_t-(a+q)\lambda_j}{2}\left[\frac{\tau'(t)}{\tau(t)} - \tau(t)\right]\right)[\rho^k_t-(a+q)\lambda_k].
    \label{eqn:symm_eqn_check}
\end{multline}
Thus, it remains to check
\begin{equation}
    \dot{\rho}^k_t - [\rho^k_t-(a+q)\lambda_k]\rho^k_t - \frac{\rho^k_t-(a+q)\lambda_k}{2}\left[\frac{\tau'(t)}{\tau(t)} - \tau(t)\right] = 0,\ \forall k\in[N],
    \label{eqn:check_reduced_eqn}
\end{equation}
which is equivalent to
\begin{equation}
    \frac{\dot{\rho}^k_t}{\rho^k_t - (a+q)\lambda_k} - \rho^k_t = \frac{1}{2}\left[\frac{\tau'(t)}{\tau(t)} - \tau(t)\right],\ \forall k\in[N].
    \label{eqn:suff_check_eqn}
\end{equation}

\medskip

\noindent\textbf{Step 3: Verifying equation~\eqref{eqn:suff_check_eqn}.}
We compute both sides of equation~\eqref{eqn:suff_check_eqn} separately. For the left-hand side, define \(g_k:[0,T]\to \R\) such that
\begin{equation}
    g_k(s) := c\sqrt{\frac{Q_1'(T-s)}{c}}e^{-(T-s)(a+q)\lambda_k}.
    \label{eqn:g_k}
\end{equation}
Straightforward calculations based on equation~\eqref{eqn:rho_veri} give
\begin{equation}
    \frac{\dot{\rho}^k_t}{\rho^k_t - (a+q)\lambda_k} - \rho^k_t = \frac{g_k'(t)}{g_k(t)} - (a+q)\lambda_k,
    \label{eqn:LHS_key_eqn}
\end{equation}
which corresponds to the \(k\)-th eigenvalue of the matrix \(-R''(T-t)[R'(T-t)]^{-1} - (a+q)L\), using equation~\eqref{eqn:first_charac}. 
 
For the right-hand side, rewrite \(\tau\) as a function of \(g_k\) using equations~\eqref{eqn:tau},~\eqref{eqn:rho_veri} and~\eqref{eqn:g_k}:
\begin{equation}
    \tau(t) = \frac{1}{N}\sum_{k=1}^N [\rho^k_t - (a+q)\lambda_k] = \frac{1}{N}\sum_{k=1}^N\frac{\lambda_kg_k(t)}{1+\lambda_k\int_t^T g_k(s)\,\ud s}.
    \label{eqn:tau_g_k}
\end{equation}
Next, rewriting equation~\eqref{eqn:tau_g_k} as the derivative of a logarithm yields
\begin{multline}
    \tau(t) = -\frac{\ud}{\ud t}\log\left[\prod_{k=1}^N \left(1 + \lambda_k\int_t^T g_k(s)\,\ud s\right)\right]^{\frac{1}{N}} \\
    = -\frac{\ud}{\ud t}\log\left[\det\left(I + cL \int_t^TR'(T-s)\,\ud s\right)\right]^{\frac{1}{N}}.
    \label{eqn:tau_as_det}
\end{multline}
Using equation~\eqref{eqn:first_charac}, it follows that
\begin{equation}
    \tau(t) = -\frac{d}{dt}\log Q_1(T-t) = \frac{Q_1'(T-t)}{Q_1(T-t)}.
    \label{eqn:tau_and_Q}
\end{equation}
Taking the derivative with respect to \(t\) on both sides results in
\begin{equation}
    \tau'(t) = \frac{[Q_1'(T-t)]^2 - Q_1(T-t)Q_1''(T-t)}{Q_1^2(T-t)}.
    \label{eqn:tau_derivative}
\end{equation}
Thus, the right-hand side of equation~\eqref{eqn:suff_check_eqn} reads
\begin{equation}
    \frac{1}{2}\left[\frac{\tau'(t)}{\tau(t)} - \tau(t)\right] = -\frac{Q_1''(T-t)}{2Q_1'(T-t)}.
    \label{eqn:RHS_key_eqn}
\end{equation}

Combining results from equations~\eqref{eqn:LHS_key_eqn} and~\eqref{eqn:RHS_key_eqn}, proving the spectral equation~\eqref{eqn:suff_check_eqn} requires demonstrating the following matrix equation:
\begin{equation}
    -R''(T-t)[R'(T-t)]^{-1} - (a+q)L = -\frac{Q_1''(T-t)}{2Q_1'(T-t)}I.
    \label{eqn:matrix_eqn}
\end{equation}
Differentiating both sides of equation~\eqref{eqn:first_charac} with respect to $t$ produces
\begin{equation}
    R''(T-t) = \frac{Q_1''(T-t)}{2\sqrt{cQ_1'(T-t)}}e^{-(T-t)(a+q)L} - \sqrt{\frac{Q_1'(T-t)}{c}}e^{-(T-t)(a+q)L}(a+q)L.
    \label{eqn:R_1_second_derivative}
\end{equation}
Combining equations~\eqref{eqn:R_1_second_derivative} and \eqref{eqn:first_charac} yields equation~\eqref{eqn:matrix_eqn}, thus
concluding the verification process and proving the second part of Theorem~\ref{thm:NE_closed_form}.


\section{Numerical Experiments}\label{sec:numerics}

In this section, we conduct several numerical experiments concerning the NE of the game introduced in Section~\ref{sec:model_setup}. Section~\ref{sec:num_veri} computes the semi-explicit equilibrium  derived in Theorem~\ref{thm:NE_closed_form} under various vertex-transitive graphs, and numerically verifies its consistency with solving the Riccati system~\eqref{eqn:Riccati}. 
Implied by Corollary~\ref{cor:time_scaling}, fictitious play is shown to be a generally applicable technique for any connected graphs with provable convergence. In Section~\ref{sec:num_FP}, we visualize this convergence procedure and quantify its convergence rate for different graphs.
Lastly, in Section~\ref{sec:num_time}, a comparison of the time complexity and the running time among previously mentioned numerical methods is presented.

\subsection{Numerical verification of Theorem~\ref{thm:NE_closed_form}}\label{sec:num_veri}

This section aims to numerically verify that the results derived in Theorem~\ref{thm:NE_closed_form} for vertex-transitive graphs are consistent with the baseline equilibrium strategies. The baseline equilibrium strategy refers to the one obtained by directly solving the Riccati system~\eqref{eqn:Riccati} numerically using the explicit Runge-Kutta method of order 8 \cite{wanner1996solving}. For cases where $G$ is a complete graph, an additional baseline is constructed using the closed-form solution presented in \cite{carmona2013mean}.
Regarding the semi-explicit equilibrium strategy, we first solve equation~\eqref{eqn:second_charac} numerically for \(R\), and then build  \(F^i\) from \(R\) using equations~\eqref{eqn:closed_form_F}--\eqref{eqn:R_to_P}.

In the numerical experiments, we fix the following model parameters:
\begin{equation}
    \sigma = 0.5,\quad a = 0.1,\quad c=1,\quad X_0 = x_0\in\R^N,
\end{equation}
where the initial state \(x_0\) is randomly sampled from a uniform distribution \(x_0^i\overset{\text{i.i.d.}}{\sim} \mathcal{U}(-1,1),\ \forall i\in[N]\). 
The remaining model parameters \((q,\EPS)\), \(N\), \(T\) and \(G\) vary and will be specified in Table~\ref{tab:num_veri}. Throughout the experiments, the relationship \(q^2 = \EPS\) is maintained and \(G\) is taken as one of the following vertex-transitive graphs:
  \begin{align}
        &KG_{5,2}: \text{the \(10\)-vertex Petersen graph},\quad  K_N: \text{the \(N\)-vertex complete graph},\\
        & C_N: \text{the \(N\)-vertex cycle graph}, \qquad \qquad  Q_k: \text{the \(2^k\)-vertex hypercube graph},\\
        & C^{s_1,\dots,s_k}_N: \text{the \(N\)-vertex circulant graph with jumps \(s_1,\ldots,s_k\).}
    \end{align}
These graphs are of great research interest on their own, and we refer readers to \cite{chung1997spectral,godsil2001algebraic} for their definitions, constructions, and properties.

For the time discretization, we partition the time horizon \([0,T]\) into \(N_T = 1000\) subintervals of equal lengths \(h := T/N_T\), and denote the discretization scheme by \(\Delta := \SET{kh:k\in\SET{0,1,\ldots,N_T-1}}\), which is the collection of all the subintervals' endpoints. The difference between the baseline and semi-explicit equilibrium is measured by maximum absolute error (MAE) and maximum relative error (MRE):
\begin{equation}
    \text{MAE}(\tilde{F};\check{F}) := \max_{t\in \Delta}\max_{i\in[N]}\NORM{\check{F}^i_t - \tilde{F}^i_t},\quad
    \text{MRE}(\tilde{F};\check{F}) := \max_{t\in \Delta}\max_{i\in[N]}\frac{\NORM{\check{F}^i_t - \tilde{F}^i_t}}{\NORM{\check{F}^i_t}},
    \label{eqn:MAE_MRE}
\end{equation}
where \(\check{F}^i\) represents the numerical solution of the Ricatti system~\eqref{eqn:Riccati}, and \(\tilde{F}^i\) denotes the semi-explicit counterpart. Small values of \(\text{MAE}\) and \(\text{MRE}\) indicate the alignment between the baseline and the semi-explicit construction, providing evidence for the correctness of Theorem~\ref{thm:NE_closed_form}. Table~\ref{tab:num_veri} provides the \(\text{MAE}\) and \(\text{MRE}\) values for various models. 

\begin{table}[ht!]
\begin{center}
\caption{Model parameters, MAE, and MRE for the examples in Section~\ref{sec:num_veri}}
\label{tab:num_veri}
\begin{tabular}{c|cccccc}
    \toprule
    \# of Players & \(N = 10\) & \(N = 30\) & \(N = 40\)  & \(N = 64\) & \(N = 128\) & \(N = 150\) \\
    \midrule
    Graph \(G\) & \(KG_{5,2}\) & \(K_{30}\) & \(C_{40}\)  & \(Q_6\) & \(Q_7\) & \(C_{150}^{11}\) \\
    Time \(T\) & \(1\) & \(1\) & \(1.5\)  & \(1.5\) & \(2\) & \(2.5\)\\
    \((q,\EPS)\) & \((0,0)\) & \((1,1)\)  & \((0,0)\) & \((2,4)\) & \((0,0)\) & \((1,1)\) \\
    \(\text{MAE}\) & \(3.88e{-6}\) & \(2.09 e{-5}\)  & \(7.12e{-6}\) & \(4.67e{-6}\) & \(1.92e{-6}\) & \(3.16e{-5}\)\\
    \(\text{MRE}\) & \(5.07e{-6}\) & \(8.96e{-5}\)  & \(1.44e{-5}\) & \(8.21e{-4}\) & \(5.30e{-6}\) & \(1.16e{-3}\)\\
    \bottomrule
\end{tabular}
\end{center}
\end{table}

Next, we visualize the equilibrium state and strategy processes. To this end, we adopt the Euler scheme to equation~\eqref{eqn:state_dynamics}  and simulate it with the feedback strategy $\alpha_t^i$ plugged in using the previously computed $\check F_t^i$ and $\tilde F_t^i$ (cf. \eqref{eqn:NE_strategy} and \eqref{eqn:transitive_NE}): 
\begin{align}
    &\hat{X}^i_{t+h} = \hat{X}^i_t + \left[a\left(\frac{1}{\sqrt{d_{v_i}}}\sum_{j:v_j\sim v_i}\frac{1}{\sqrt{d_{v_j}}}\hat{X}^j_t - \hat{X}^i_t\right) + \hat{\alpha}^i_t\right]h + \sigma\sqrt{h}\xi^i_t,\ \xi^i_t\overset{\text{i.i.d.}}{\sim} \mathcal{N}(0,1),\\
    & \hat{\alpha}^i_t = -qe_i\transpose L\hat{X}_t - e_i\transpose F^i_t\hat{X}_t,\quad \forall i\in[N],\ \forall t\in \Delta,\ \text{where}\ F^i= \check{F}^i\ \text{or}\ \tilde{F}^i.
\end{align}
In Figures~\ref{fig:complete_check}--\ref{fig:cycle_check}, trajectories of the equilibrium state process \(\hat{X}\) and the equilibrium strategy process \(\hat{\alpha}\) from four randomly selected players are presented for complete and cycle graphs, with the choice of time discretization \(N_T = 50\), time horizon \(T = 1\), and model parameters \((q,\EPS) = (1,1)\).

\begin{figure}
    \centering
   \includegraphics[width = \textwidth, trim = {0 40 0 0}, clip]{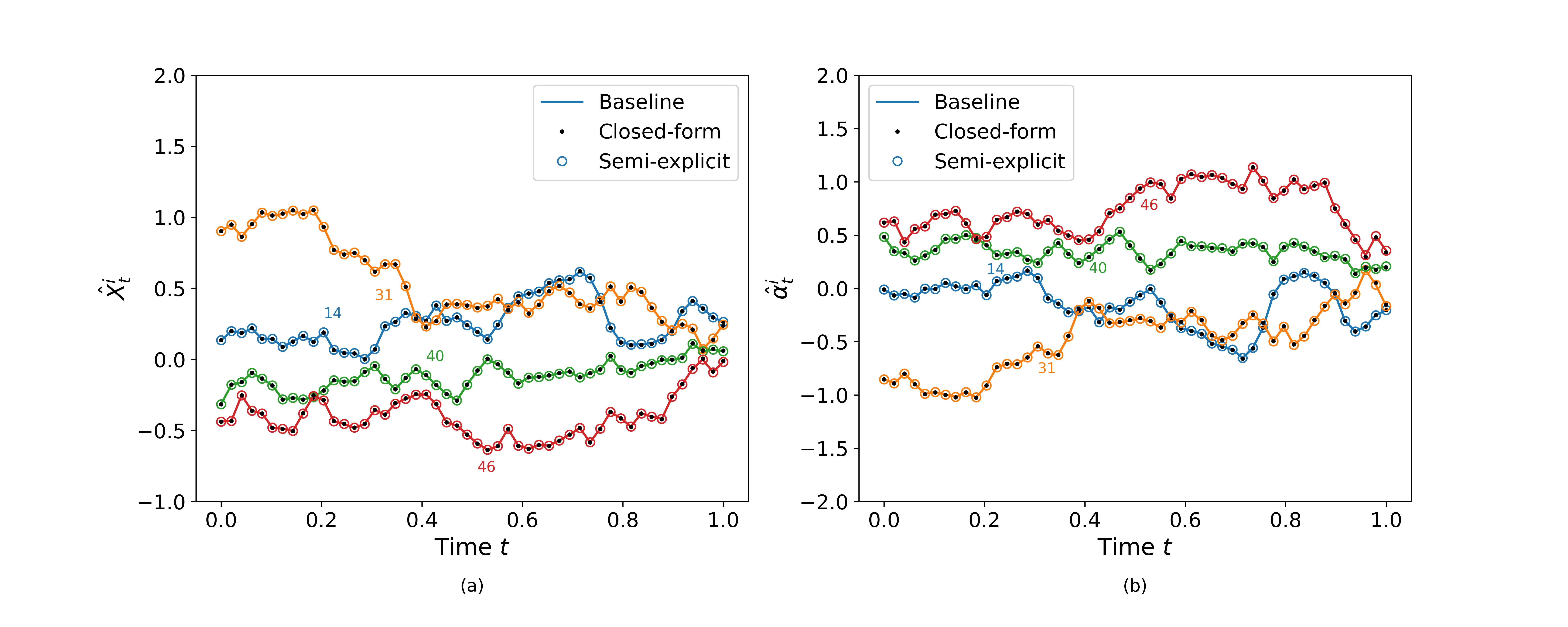}
    \caption{Comparisons of equilibrium state (left panel) and strategy (right panel) trajectories for \(N = 50\) players in the linear-quadratic game on the complete graph \(G = K_{50}\).
    In both panels, the colored solid lines represent the baseline solution (by numerically solving the Riccati system~\eqref{eqn:Riccati}), the colored circles are obtained by numerically solving the semi-explicit solution in Theorem~\ref{thm:NE_closed_form}, and the black dots are computed by the closed-form solution given in \cite{carmona2013mean}.
    For the sake of clarity, only trajectories of four randomly selected players are plotted.
    }
    \label{fig:complete_check}
\end{figure}

\begin{figure}
    \centering
    \includegraphics[width = \textwidth, trim = {0 35 0 0}, clip]{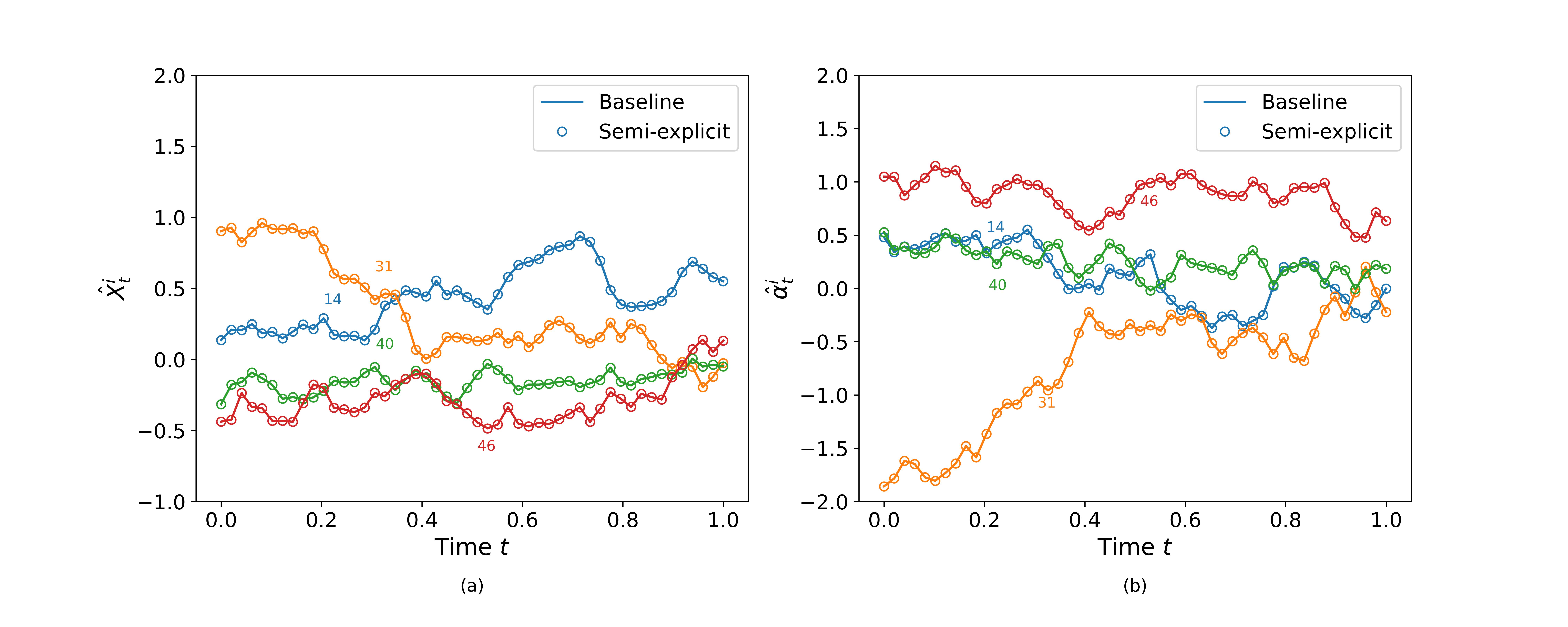}
    \caption{Comparisons of equilibrium state (left panel) and strategy (right panel) trajectories for \(N = 50\) players in the linear-quadratic game on the cycle graph \(G = C_{50}\).
 In both panels, the colored solid lines represent the baseline solution (by numerically solving the Riccati system~\eqref{eqn:Riccati}), and the colored circles are obtained by numerically solving the semi-explicit solution in Theorem~\ref{thm:NE_closed_form}.
    For the sake of clarity, only trajectories of four randomly selected players are plotted }
    \label{fig:cycle_check}
\end{figure}

\subsection{Numerical convergence of fictitious play}\label{sec:num_FP}

Next, we examine the convergence of FP, visualize the convergence procedure, and compare the rate of convergence for different graphs. Our experiment identifies a contraction-mapping-like structure, linking the rate of convergence to the graph structure. Moreover, as will be observed below, condition~\eqref{eqn:FP_final_condition} for Corollary~\ref{cor:time_scaling} is sufficient but not necessary for the convergence of FP. This further validates the general applicability of FP as a state-of-the-art numerical technique solving for the Nash equilibrium.

In the numerical experiments, we fix the following model parameters:
\begin{equation}
    N = 50,\quad T = 1,\quad q = 0.5,\quad \EPS = 1,\quad \sigma = 0.5,\quad a = 0.1,\quad c=1,\quad X_0 = x_0\in\R^N,
\end{equation}
where the initial state \(x_0\) is randomly sampled from a uniform distribution \(x_0^i\overset{\text{i.i.d.}}{\sim} \mathcal{U}(-1,1),\ \forall i\in[N]\).
Throughout the experiments, \(G\) is taken as one of the following graphs, with the specific choice given in Tables~\ref{tab:num_FP_deter}--\ref{tab:num_FP_rand}:
  \begin{align}
        &K_N: \text{the \(N\)-vertex complete graph},\qquad C_{N}: \text{the \(N\)-vertex cycle graph},\\
        & S_N: \text{the \(N\)-vertex star graph}, \qquad \qquad \\
        &K_{m,m}: \text{the \(2m\)-vertex complete bipartite graph with partitions of size \(m\) and \(m\)},\\
        & \text{RSG}^s_N: \text{the \(N\)-vertex random sparse graph with edge density \(s\)},\\
        &\text{RMJ}^d_N: \text{the \(N\)-vertex \(d\)-regular random Ramanujan graph.}
    \end{align}
Among the graphs mentioned above, \(\text{RSG}^s_N\) and \(\text{RMJ}^d_N\) are random graphs, while the others are deterministic. For the construction of random sparse graphs, random edges are added to a randomly generated spanning tree until the sparsity measure of the graph satisfies certain conditions. Here, we select the edge density \(s\) as the sparsity measure, defined as \(s:= |E|/\binom{|V|}{2}\), which is the ratio between the number of edges present in the graph and the maximum possible number of edges the graph can contain. On the other hand, the \(d\)-regular Ramanujan graph \cite{lubotzky1988ramanujan} is a randomly generated regular expander graph with common degree \(d\), which essentially saturates the Alon–Boppana bound \cite{hoory2006expander}. Ramanujan graphs are important examples of highly connected sparse graphs \cite{chung1997spectral}, demonstrating the possible coexistence of two seemingly contradictory graph properties. Therefore, we include them in the experiments. 

Starting with initial strategies \(\check{F}^{i,0}_t = 0,\ \forall i\in[N]\), \(t\in\Delta\), we compute \(\{\check{F}^{i,k}_t:i\in[N],t\in\Delta\}\) at stage \(k\in\N\) by numerically solving system~\eqref{eqn:recursive_Riccati_FP} iteratively with a time discretization level \(N_T = 1000\).
We perform \(N_{\text{round}} = 10\) rounds of FP for each numerical example and assess the convergence of FP by calculating \(\text{MAE}(\check{F}^{\cdot,N_{\text{round}}};\check{F})\) and \(\text{MRE}(\check{F}^{\cdot,N_{\text{round}}};\check{F})\), both of which have been previously defined in equation~\eqref{eqn:MAE_MRE}.
The numerical results are organized in Tables~\ref{tab:num_FP_deter}--\ref{tab:num_FP_rand}:
Table~\ref{tab:num_FP_deter} for deterministic graphs and Table~\ref{tab:num_FP_rand} for three different samples of random graphs \(\text{RSG}^s_N\) and \(\text{RMJ}^d_N\).

\begin{table}[ht!]
\begin{center}
\caption{MAE and MRE for LQ games on deterministic graphs in Section~\ref{sec:num_FP}}
\label{tab:num_FP_deter}
\begin{tabular}{c|cccc}
    \toprule
    Problem & \(G = K_{50}\) & \(G = C_{50}\) & \(G = S_{50}\) & \(G = K_{25,25}\) \\
    \midrule
    \(\text{MAE}\) & \(2.19e{-3}\) & \(8.58e{-2}\)  & \(1.50e{-1}\) & \(1.76e{-2}\) \\
    \(\text{MRE}\) & \(3.21e{-3}\) & \(1.02e{-1}\)  & \(1.69e{-1}\) & \(2.26e{-2}\) \\
    \bottomrule
\end{tabular}
\end{center}
\end{table}

\begin{table}[ht!]
\begin{center}
\caption{MAE and MRE for LQ games on random graphs in Section~\ref{sec:num_FP}}
\label{tab:num_FP_rand}
  \begin{tabular}{c|cccccc}
    \toprule
    \multirow{2}{*}{Problem} &
      \multicolumn{3}{c}{\(G = \text{RSG}_{50}^{0.125}\)} &
      \multicolumn{3}{c}{\(G = \text{RMJ}_{50}^{6}\)} \\
    & Sample 1 & Sample 2 & Sample 3 & Sample 1 & Sample 2 & Sample 3\\
      \midrule
    MAE & \(5.62e{-2}\) & \(4.60e{-2}\) & \(4.58e{-2}\) & \(3.69e{-2}\) & \(3.70e{-2}\) & \(3.68e{-2}\)\\
    MRE & \(7.23e{-2}\) & \(6.03e{-2}\) & \(5.98e{-2}\) & \(4.82e{-2}\) & \(4.82e{-2}\) & \(4.83e{-2}\)\\
    \bottomrule
  \end{tabular}
\end{center}
\end{table}

The MAEs and MREs shown in Tables~\ref{tab:num_FP_deter}--\ref{tab:num_FP_rand} are consistently small, implying the numerical convergence of FP, even when the model parameters are deliberately chosen not to satisfy condition~\eqref{eqn:FP_final_condition} of Corollary~\ref{cor:time_scaling}. We observe that the magnitude of errors varies for different graph structures. To better compare the convergence rates, we plot \(\log\text{MAE}(\check{F}^{\cdot,k-1};\check{F}^{\cdot,k})\)  \textit{vs.} the FP stage index \(k\), as shown in Figure~\ref{fig:FP_graph}. The log-MAE curves are nearly straight lines with slopes depending on the graph structures. This implies 
a contraction-mapping-like structure
\begin{equation}
    \text{MAE}(\check{F}^{\cdot,k};\check{F}^{\cdot,k+1})\approx e^{C_G}\cdot\text{MAE}(\check{F}^{\cdot,k-1};\check{F}^{\cdot,k}),\ \forall k\in\N,
    \label{eqn:contraction_like}
\end{equation}
where \(C_G \in (-\infty,0)\) is a constant that depends on the graph \(G\). A simple linear regression gives   
\begin{align*}
    &C_{K_{50}} = -4.55, \; && C_{C_{50}} = -1.94, \; &&C_{S_{50}} = -2.01, \; \\
    &C_{K_{25, 25}} = -2.61, \; &&C_{\text{RSG}^{0.125}_{50}} = -2.06, \;&& C_{\text{RMJ}^6_{50}} = -2.07.
\end{align*}
For the purpose of comparison, we provide below the norms of the graph Laplacians \(L_G\) used in the numerical experiments, denoted as \(N_G:=\NORM{L_G}\in[0,2]\), that
\begin{align*}
   & N_{K_{50}} = 1.02, \;  && N_{C_{50}} = 2.00, \; &&N_{S_{50}} = 2.00,  \; \\
   & N_{K_{25, 25}} = 2.00, \; && N_{\text{RSG}^{0.125}_{50}} = 1.68, \; &&N_{\text{RMJ}^6_{50}} = 1.70.
\end{align*}
A pairwise comparison between complete graphs and the other graphs shows that a large norm of the graph Laplacian possibly results in a slower rate of convergence of FP, which is consistent with the insights provided in Remark~\ref{rem:rate_graph}.
However, the norm of the graph Laplacian is not the only factor that affects the rate of convergence of FP.
For instance, the complete bipartite graph has \(N_{K_{25,25}} = 2.00\), taking the largest value on $[0,2]$, yet FP converges faster on complete bipartite graphs compared to random sparse graphs and Ramanujan graphs.
Consequently, accurately quantifying the rate of convergence of FP in terms of the graph structure requires extra work and is left for future studies.

\begin{figure}
    \centering
    \includegraphics[width = 0.8\textwidth]{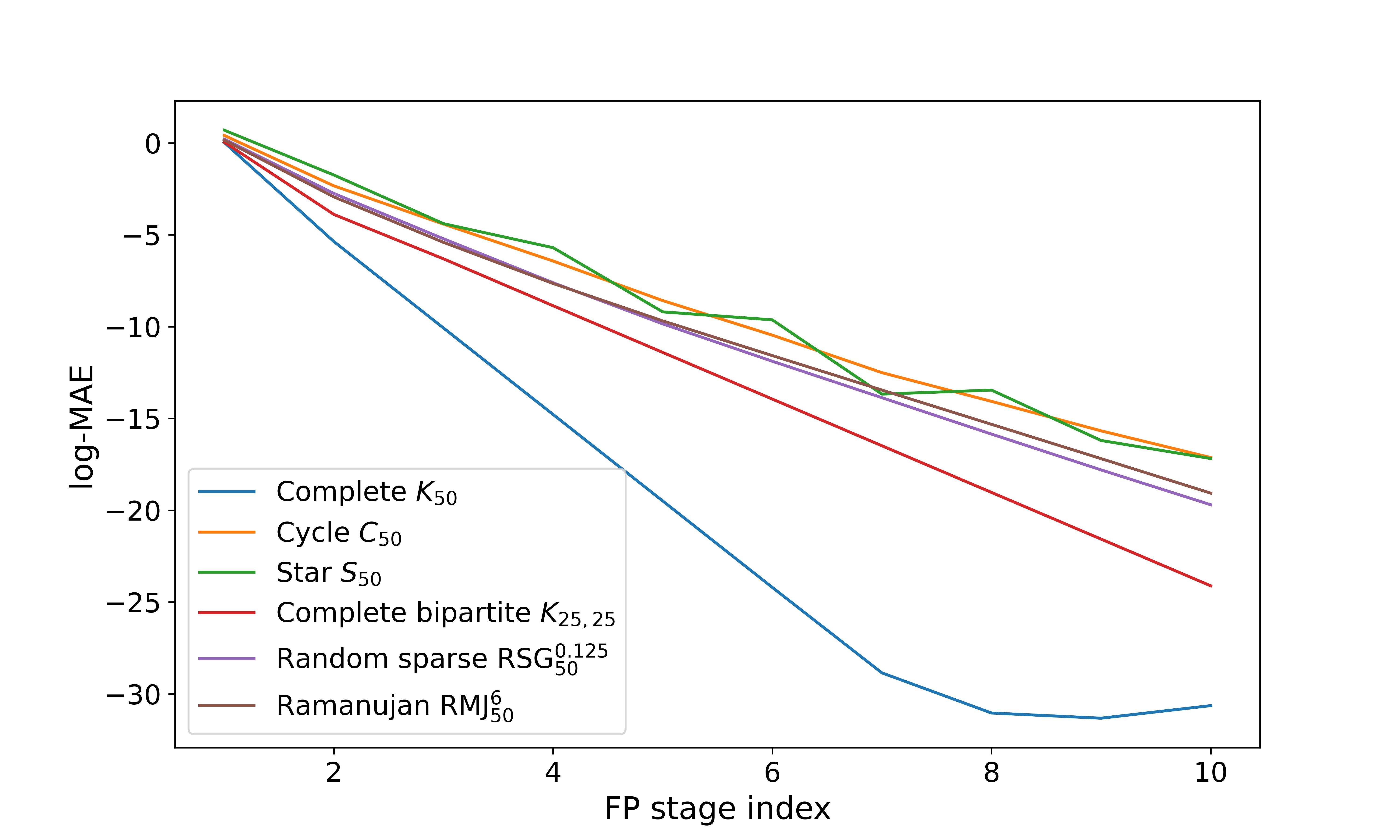}
    \caption{Comparisons of log-MAE curves for the linear-quadratic game with \(N = 50\) players on different graphs. The behavior at the tail of the log-MAE curves for the complete graphs is due to the numerical round-off error
   }
    \label{fig:FP_graph}
\end{figure}

\subsection{The time complexity analysis and running time comparison}\label{sec:num_time}

So far, we have implemented three numerical methods for solving the NE of linear-quadratic games on graphs: (i) the baseline construction by directly solving the Riccati system~\eqref{eqn:Riccati}; (ii) the semi-explicit solution based on Theorem~\ref{thm:NE_closed_form}; and (iii) fictitious play by iteratively solving system~\eqref{eqn:recursive_Riccati_FP}.
In this section, we conduct a comparative study of the efficiency of these methods by investigating both the theoretical time complexity and the numerical running time, aiming to provide guidance on the trade-off between the efficiency and the generality of these numerical methods.

\medskip
\noindent\textbf{The time complexity analysis.} We first discuss the time complexity of the three numerical methods, and the comparison intends to show the efficiency of constructing the semi-explicit NE using Theorem~\ref{thm:NE_closed_form}.

For illustration purposes, the analysis of time complexity uses the forward Euler method as the numerical ODE solver.
In the baseline construction for solving the Riccati system~\eqref{eqn:Riccati}, each Euler step computes \((\check{F}^1_{t+h},\ldots,\check{F}^N_{t+h})\) based on \((\check{F}^1_{t},\ldots,\check{F}^N_{t})\).
The calculation of the derivative \(\dot{\check{F}}^i_{t}\) requires \(O(N^3)\) calculations, so a single Euler step costs \(O(N^4)\), and constructing the entire baseline solution requires \(O(N_TN^4)\).
Meanwhile, only one matrix-valued ODE is numerically solved during the semi-explicit construction, so the whole procedure only costs \(O(N_TN^3)\).

Lastly, we analyze the time complexity of FP.
To conduct a fair comparison, the number of rounds of FP, denoted by \(N_{\text{round}}\), should be such that \(\check{F}^{\cdot,N_{\text{round}}}\) has approximately the same order of error as the baseline construction \(\check{F}\). Since the Euler method has convergence order one \cite{kress2012numerical}, i.e., \(\text{MAE}(\check{F},F) = O(h)\) as \(h\to 0\) with a time discretization scheme that has step size \(h\), \(N_{\text{round}}\) should be such that \(\text{MAE}(\check{F}^{\cdot,N_{\text{round}}},F) = O(h)\).
Due to the contraction-mapping-like structure~\eqref{eqn:contraction_like} empirically observed, \(N_{\text{round}} = O(\log \frac{1}{h})\) and the entire procedure of FP costs \(O(N_T\log N_TN^4)\). The above discussion is summarized in Table~\ref{tab:time_complex} below.

\begin{table}[ht!]
\begin{center}
\caption{Time complexity analysis of three numerical methods. All constructions have a maximum absolute error of the same order.}
\label{tab:time_complex}
\begin{tabular}{c|ccc}
    \toprule
    Construction Methods & Baseline & Semi-explicit & Fictitious play \\
    \midrule
    Time Complexity & \(O(N_TN^4)\) & \(O(N_TN^3)\)  & \(O(N_T\log N_TN^4)\) \\
    \bottomrule
\end{tabular}
\end{center}
\end{table}

\medskip
\noindent\textbf{Running time.}
Next, we compare the running time of the three methods under the following choice of model parameters:
\begin{equation}
    G = K_N,\quad T = 1,\quad q = 1,\quad \EPS = 1,\quad \sigma = 0.5,\quad a = 0.1,\quad c=1,\quad X_0 = x_0\in\R^N,
\end{equation}
where the initial state \(x_0\) is randomly sampled from a uniform distribution \(x_0^i\overset{\text{i.i.d.}}{\sim} \mathcal{U}(-1,1),\ \forall i\in[N]\).
The remaining model parameter \(N\) varies and will be specified in Table~\ref{tab:num_running}. To ensure a fair comparison, we fix the error tolerance level at \(\text{tol} = 10^{-5}\) as indicated by Table~\ref{tab:num_veri}, and stop the FP procedure at stage \(k\) once \(\text{MAE}(\check{F}^{\cdot,k-1},\check{F}^{\cdot,k}) < \text{tol}\).
The numerical outcomes are presented in Table~\ref{tab:num_running} and align with the time complexity analysis shown in Table~\ref{tab:time_complex}.

\begin{table}[!ht]
\begin{center}
\caption{Running time comparison of three numerical methods: (i) Baseline from directly solving the Riccati system~\eqref{eqn:Riccati}; (ii) Semi-explicit based on Theorem~\ref{thm:NE_closed_form}; and (iii) Fictitious play by iteratively solving system~\eqref{eqn:recursive_Riccati_FP}. 
For each value of \(N\), the exact running time is measured in minutes, while the percentage is calculated relative to the baseline construction.}

\label{tab:num_running}
\begin{tabular}{c|cccccc}
    \toprule
    Running time (mins) & \multicolumn{2}{c}{\(N = 32\)} &
    \multicolumn{2}{c}{\(N = 64\)} &
    \multicolumn{2}{c}{\(N = 128\)} \\
    \midrule
    Baseline & \(0.03\) & \(100\%\) & \(0.29\) & \(100\%\)  & \(2.50\) & \(100\%\) \\
    Semi-explicit & \(0.01\) & \(33\%\) & \(0.05\) & \(17\%\) & \(0.25\) & \(10\%\) \\
    Fictitious play & \(0.18\) & \(600\%\) & \(1.68\) & \(579\%\)  & \(10.01\) & \(400\%\) \\
    \bottomrule
\end{tabular}
\end{center}
\end{table}

\section{Conclusion and Future Studies}\label{sec:conclusion_future}

In this paper, we analyzed a class of linear-quadratic games on graphs~\eqref{eqn:state_dynamics}--\eqref{def:J}, where the players' interactions are characterized by the graph Laplacian. 
For generic graphs, we employed fictitious play, one of the most important numerical techniques to compute Nash equilibria, and provided its convergence proof. 
For vertex-transitive graphs with nice symmetry properties, we developed a semi-explicit characterization of the Nash equilibrium and proved its well-posedness. Additionally,  numerical experiments were conducted, providing numerical verifications for our previously proved theorems and a comparison among the numerical methods. The findings consist of: (i) The semi-explicit characterization (Theorem~\ref{thm:NE_closed_form}) is the most efficient method, but it only works for vertex-transitive graphs when the model parameters satisfy \(\EPS = q^2\); (ii) The baseline method has intermediate efficiency and requires the linear-quadratic structure of the game, but it is applicable to any connected graphs with no restrictions on the model parameters; and (iii) Fictitious play is the least efficient method, but it can be parallelized and applied beyond linear-quadratic games.

As machine learning methods have become more powerful, there has been an increasing trend to solve control problems and games numerically using deep learning and reinforcement learning techniques. We refer readers to the survey \cite{hu2023recent} for a general overview of these methods. In particular, direct parameterization \cite{han2016deep} and Deep BSDE \cite{han2017deep} are two popular methods for control problems and have been generalized to various extents. Furthermore, \cite{han2020deep,hu2019deep} use the deep version of fictitious play as a fundamental strategy for solving the NE.
Along with this research direction, we are developing graph-based deep-learning algorithms for solving the Nash equilibrium of large sparse network games. This effort is mainly motivated by the theoretical results and numerical findings in this work. The algorithm accelerates training by reducing the number of parameters without significantly losing the accuracy. 
The details of the algorithm and the numerical outcomes will be presented in our upcoming work \cite{hu2024deep}.

Finally, we point out that several directions are worth further investigation.
For instance, although the global convergence of fictitious play has been observed in numerical experiments even when all model parameters take relatively large values, we were only able to provide the proof of convergence (cf. Corollary~\ref{cor:time_scaling}) under a smallness condition.  Further relaxation of the technical assumption would be valuable.
Additionally, regarding the rate of convergence, fictitious play numerically exhibits a graph-dependent rate as seen in Section~\ref{sec:num_FP}, which has not been accurately characterized in Corollary~\ref{cor:time_scaling}. These studies on the convergence of fictitious play are left to future investigation. Separately, for linear-quadratic games on vertex-transitive graphs, we encounter technical difficulties when developing the semi-explicit characterization for the NE when \(\EPS \neq q^2\), which remains unsolved. The game equilibrium and asymptotic correlations of the limiting regime $(N \to \infty)$, in the spirit of  \cite{lacker2022case}, are also interesting topics for future work. 
In addition, analyzing our generic model under more restrictive information structures, e.g., a partially observable setting where agents only access local information within its graph neighborhood and maintain beliefs for unobservable information, would be an interesting direction for future work.
Along this research direction, numerical works have been done in the field of multi-agent reinforcement learning under a collaborative game setting \cite{chu2020mult,zhang2018full}.
Theoretical works have been done under a collaborative setting in the mean-field regime when a network of states (but not players) is given \cite{gu2025mean}.
To our best knowledge, there is no theoretical analysis for the setting of a competitive game on graphs with local information structures.
Such a setting provides a more realistic modeling of strategic interactions on graphs, compared to the complete information case discussed in this paper, and deserves future investigations.

\section*{Acknowledgement}
R.H. was partially supported by a grant from the Simons Foundation (MP-TSM-00002783), and the ONR grant under \#N00014-24-1-2432.

\newcommand{\etalchar}[1]{$^{#1}$}
\providecommand{\bysame}{\leavevmode\hbox to3em{\hrulefill}\thinspace}
\providecommand{\MR}{\relax\ifhmode\unskip\space\fi MR }
\providecommand{\MRhref}[2]{%
  \href{http://www.ams.org/mathscinet-getitem?mr=#1}{#2}
}
\providecommand{\href}[2]{#2}

\begin{appendices}
\section{Heuristic Derivations}\label{sec:heuristic}

This appendix provides detailed heuristic derivations of how equations~\eqref{eqn:closed_form_F},  together with \eqref{eqn:second_charac}--\eqref{eqn:Q} and \eqref{eqn:R_to_P}, are obtained. This supplements the steps briefly outlined in Section~\ref{sec:roadmap}.

During the heuristic derivations below, several assumptions must be made. For a smoother presentation flow, these assumptions (\hyperref[assu:commute]{A.1})--(\hyperref[assu:power_series]{A.4}) are summarized later in Remark~\ref{rem:assumptions}, and are referred to during the heuristic derivation. However, we emphasize that verifying these assumptions is not necessary due to the existence of the verification procedure presented in Section~\ref{subsec:verification}.

Let  \(P:[0,T]\to \R^{N\times N}\) be a matrix-valued function that is defined in terms of  \(F^1_t,\cdots,F^N_t\):
    \begin{equation}
        P_t := \sum_{k=1}^N \left[(a+q)L + F^k_t\right]e_ke_k\transpose = (a+q)L + \sum_{k=1}^N F^k_te_ke_k\transpose.
        \label{eqn:defn_P}
    \end{equation}
Under assumption~(\hyperref[assu:invertible]{A.1}), the Riccati system~\eqref{eqn:Riccati} can be rewritten in terms of \(P_t\):
\begin{equation}
    \dot{F}^i_t - P_tF^i_t - F^i_tP_t +(\EPS-q^2)Le_ie_i\transpose L + F^i_te_ie_i\transpose F^i_t = 0,\quad F^i_T = cLe_ie_i\transpose L.
    \label{eqn:Riccati_P}
\end{equation}
Treating $P$ as given, the solution to equation~\eqref{eqn:Riccati_P} can be represented in terms of \(P_t\), denoted as \(F^1_t(P),\cdots,F^N_t(P)\). Substituting \(F^1_t(P),\cdots,F^N_t(P)\) into equation~\eqref{eqn:defn_P} provides an equation solely for $P$:
\begin{equation}
    P_t = (a+q)L + \sum_{k=1}^N F^k_t(P)e_ke_k\transpose.
\end{equation}
Once a fixed point $P^\ast$ is obtained from the above equation, $F_t^k(P^\ast)$ naturally provides the solution.


On the other hand, vertex-transitive graphs naturally exhibit rich symmetry, which can eliminate many quantities' dependence on the player index $i$. To facilitate the derivations, we use the regular representation of the graph automorphism group, which is defined below. 


\begin{defn}
Recall from Definition~\ref{defn:transitive} that $\text{Aut}(G)$ denotes the set of automorphisms of \(G\). 
    For any \(\varphi\in \text{Aut}(G)\), the associated regular representation \(R_\varphi \in\R^{N\times N}\) is defined as an invertible matrix such that \(\forall i\in [N]\),
    \begin{equation}
        R_\varphi e_i = e_{\varphi(i)}.
    \end{equation}
    \label{def:group_rep}
\end{defn}

Immediately from Definition~\ref{def:group_rep}, one has  $R_\varphi R_\psi = R_{\varphi\circ\psi}$, $R^{-1}_\varphi = R_{\varphi^{-1}} = R_\varphi\transpose$, $\forall \varphi,\psi\in\text{Aut}(G)$.
The following lemma outlines key properties for leveraging the symmetry of vertex-transitive graphs.

\begin{lem}[{\cite[Lemma~4.1]{lacker2022case}}]  \label{lemma:group_rep}

    For a vertex-transitive graph \(G\) with the vertex set \(V = [N]\), the following statements are true:
\begin{enumerate}[(i)]
    \item \(L\) commutes with \(R_\varphi\) for every \( \varphi\in\mathrm{Aut}(G)\).
    
    \item If \(Y\in\R^{N\times N}\) commutes with \(R_\varphi\) for every \( \varphi\in\mathrm{Aut}(G)\), then \(Y_{ii} = \frac{1}{N}\Tr(Y)\),  \(\forall i\in[N]\).
    
    \item If \(Y^1,\cdots,Y^N\in\R^{N\times N}\) are such that \(R_\varphi Y^i = Y^{\varphi(i)}R_\varphi\) holds for every \(\varphi\in\mathrm{Aut}(G)\) and \( i\in[N]\), then \(Y_{ii}^i = Y_{jj}^j\), \(\forall i,j\in[N]\).
\end{enumerate}
  
\end{lem}

\subsection{Solving \texorpdfstring{\(F^i\)}{} in terms of \texorpdfstring{\(P\)}{}}\label{app:s1}

Our first step is to solve the Riccati equation~\eqref{eqn:Riccati_P} for $F_t^i$, treating $P$ as given, and represent the solution in terms of \(P\). 
To this end, we consider the linear system for $(Y_t^i, \Lambda_t^i)$:
\begin{equation}
    \begin{bmatrix}
    \dot Y^i_t \\
    \dot \Lambda^i_t
    \end{bmatrix}
    = 
    \begin{bmatrix}
        P_t & -(\EPS-q^2)Le_ie_i\transpose L\\
        e_ie_i\transpose & -P_t 
    \end{bmatrix}
     \begin{bmatrix}
         Y^i_t \\
    \Lambda^i_t
    \end{bmatrix},
    \label{eqn:tran_underlying}
\end{equation}
with terminal conditions \(Y^i_T = cLe_ie_i\transpose L\), \(\Lambda^i_T = I\).
If the solution to equation~\eqref{eqn:tran_underlying} is found and \(\Lambda^i_t\) is invertible for any \(t\in[0,T]\), then $F^i_t = Y^i_t(\Lambda^i_t)^{-1}$ is the solution to equation~\eqref{eqn:Riccati_P} \cite{blanes2000approximate}. Under the condition \(q^2 = \EPS\), the equation for $Y_t^i$ becomes decoupled, resulting in \(Y^i_t = \tilde{P}_tY^i_T\), where
\begin{equation}
    \tilde{P}_t = e^{-\int_t^T P_s\,ds}.
    \label{eqn:tilde_P}
\end{equation}
By substituting \(Y^i_t = \tilde{P}_tY^i_T\) into equation~\eqref{eqn:tran_underlying}, we obtain an ODE for $\Lambda_t^i$: 
\(\dot{\Lambda}^i_t = e_ie_i\transpose\tilde{P}_tY^i_T - P_t\Lambda^i_t\). The solution to this equation:
\(\Lambda^i_t = \tilde{P}^{-1}_t\left(\Lambda^i_T - \int_t^T \tilde{P}_se_ie_i\transpose\tilde{P}_sY^i_T\,\ud s\right)\).
Therefore, incorporating the terminal conditions, the underlying linear system yields the following solution:
\begin{equation}
    \begin{cases}
    Y_t^i = c\tilde{P}_tLe_ie_i\transpose L\\
    \Lambda_t^i = \tilde{P}^{-1}_t\left(I - c\int_t^T \tilde{P}_se_ie_i\transpose\tilde{P}_sLe_ie_i\transpose L\,\ud s\right)
    \end{cases}.
    \label{eqn:Y_Lambda}
\end{equation}

Assuming~(\hyperref[assu:tilde_P_commute]{A.2}) and applying Lemma~\ref{lemma:group_rep}(ii) gives,
\begin{equation}
    e_i\transpose\tilde{P}_sLe_i = \frac{1}{N}\Tr(\tilde{P}_sL).
    \label{eqn:P_to_trace}
\end{equation}
This simplifies \(\Lambda_t^i\) to:
\begin{equation}
    \Lambda_t^i = \tilde{P}^{-1}_t\left(I - c\int_t^T \frac{\Tr(\tilde{P}_sL)}{N}\tilde{P}_se_ie_i\transpose L\,\ud s\right).
    \label{eqn:Lambda}
\end{equation}
Under assumption~(\hyperref[assu:invertible]{A.3}), the solution to equation~\eqref{eqn:Riccati_P} is provided as
\begin{equation}
    F^i_t = c\tilde{P}_tLe_ie_i\transpose L\left(I - c\int_t^T \frac{\Tr(\tilde{P}_sL)}{N}\tilde{P}_se_ie_i\transpose L\,\ud s\right)^{-1}\tilde{P}_t.
    \label{eqn:solution_Riccati_P}
\end{equation}
Up until now, we've derived a solution \(F^i_t\) expressed in terms of \(P\).

\begin{rem}
    The condition $q^2 = \varepsilon$ ensures that the underlying linear system for $(Y^i_t, \Lambda^i_t)$ is decoupled.
    Without this condition, we need to solve the full system~\eqref{eqn:tran_underlying}, for which the existence of a closed-form solution remains unclear since, when $t \neq s$:   
    \begin{equation*}
         \begin{bmatrix}
        P_t & -(\EPS-q^2)Le_ie_i\transpose L\\
        e_ie_i\transpose & -P_t 
    \end{bmatrix}  \text{ does not commute with }  \begin{bmatrix}
        P_s & -(\EPS-q^2)Le_ie_i\transpose L\\
        e_ie_i\transpose & -P_s 
    \end{bmatrix}. \end{equation*}
    
    \label{rem:q^2=eps}
\end{rem}

\subsection{Expressing \texorpdfstring{\(P\)}{} in terms of \texorpdfstring{\(\eta\)}{}}\label{app:s2}

To further simplify equation~\eqref{eqn:solution_Riccati_P}, we  define
\begin{equation}
    \Sigma^i_t := \int_t^T \frac{\Tr(\tilde{P}_sL)}{N}\tilde{P}_se_ie_i\transpose L\,\ud s,
    \label{eqn:Sigma}
\end{equation}
and
\begin{equation}
    \eta^i_t := e_i\transpose L\left(I - c\Sigma^i_t\right)^{-1}\tilde{P}_te_i.
    \label{eqn:eta}
\end{equation} 
Thus $F_t^i$ admits the representation \(F^i_t = c\tilde{P}_tLe_ie_i\transpose L\left(I - c\Sigma^i_t\right)^{-1}\tilde{P}_t\) and satisfies \(F^i_te_i = c\eta^i_t\tilde{P}_tLe_i\).
Under assumption~(\hyperref[assu:tilde_P_commute]{A.2}), using Lemma~\ref{lemma:group_rep}(i) and equation~\eqref{eqn:Sigma} leads to  
\(R_\varphi \Sigma^i_t = \Sigma^{\varphi(i)}_tR_\varphi\), \(\forall \varphi\in \mathrm{Aut}(G)\).
With the additional assumption~(\hyperref[assu:power_series]{A.4}), we find:
\begin{equation}
    R_\varphi L\left(I - c\Sigma^i_t\right)^{-1} = L\left(I - c\Sigma^{\varphi(i)}_t\right)^{-1}R_\varphi,\ \forall \varphi\in \mathrm{Aut}(G).
    \label{eqn:deduce_eta_same}
\end{equation}
By Lemma~\ref{lemma:group_rep}(iii), \(\eta^i_t\) becomes independent of the player index $i$. From this point, we will use the notation $\eta_t$ instead of $\eta_t^i$, as this quantity is shared by all players. 

We now proceed to derive an expression for $P_t$ in terms of $\eta_t$. 
Substituting the relation \(F^i_te_i = c\eta_t\tilde{P}_tLe_i\)
into equation~\eqref{eqn:defn_P} yields
\begin{equation}
    P_t 
    = (a+q)L + c\eta_t \tilde{P}_t L,
    \label{eqn:P_tilde_P}
\end{equation}
where $\tilde P$ is defined as in equation~\eqref{eqn:tilde_P}.
It remains to solve the above equation for $P_t$. 

Under assumption~(\hyperref[assu:commute]{A.1}), any two symmetric matrices in the collection \(\{L\}\cup \{P_t,\tilde{P}_t:t\in[0,T]\}\) commute, which implies that this collection of matrices can be diagonalized simultaneously. 
Therefore, one can work with the spectra of the matrices rather than directly with the matrices themselves, providing greater convenience for calculations. Let  \(\rho^1_t,\ldots,\rho^N_t\) be the eigenvalues of \(P_t\) and \(\lambda_1,\ldots,\lambda_N\) be the eigenvalues of \(L\), 
equation~\eqref{eqn:P_tilde_P} becomes
\begin{equation}
    \rho^k_t = (a+q)\lambda_k + c\eta_t\lambda_ke^{-\int_t^T \rho^k_s\,\ud s},\ \forall k\in[N].
    \label{eqn:rho_Riccati}
\end{equation}
Multiplying both sides by \(e^{-(T-t)(a+q)\lambda_k}\) and integrating with respect to \(t\) over $[0,T]$ gives
\begin{equation}
    e^{\int_t^T \rho^k_s-(a+q)\lambda_k \,\ud s} = 1 + c\lambda_k \int_t^T \eta_s e^{-(T-s)(a+q)\lambda_k}\,\ud s.
    \label{eqn:calc_rho}
\end{equation}
Taking the logarithm on both sides and differentiating with respect to \(t\) yields
\begin{equation}
    \rho^k_t = (a+q)\lambda_k + \frac{c\lambda_k \eta_te^{-(T-t)(a+q)\lambda_k}}{1 + c\lambda_k \int_t^T \eta_s e^{-(T-s)(a+q)\lambda_k}\,\ud s}.
    \label{eqn:rho}
\end{equation}
Rewriting the above equation in matrix form yields the representation
\begin{equation}
    P_t = (a+q)L + c\eta_tLe^{-(T-t)(a+q)L}\left(I+\int_t^T c\eta_se^{-(T-s)(a+q)L}\,\ud s\; L\right)^{-1}.
    \label{eqn:rep_P_use_eta}
\end{equation}

The above expression for $P_t$ in terms of $\eta_t$ is not entirely satisfactory because $\eta$, as defined in~\eqref{eqn:eta}, depends on $\tilde P_t$, and therefore on $P_t$. Consequently, equation~\eqref{eqn:rep_P_use_eta} does not provide an explicit construction for $P_t$. The next step (in Section~\ref{sec:eta}) is to find a way to characterize and construct $\eta$ based solely on model parameters. 
Once this is achieved, equation~\eqref{eqn:rep_P_use_eta} can be used to construct $P_t$, as detailed in Section~\ref{sec:close_the_loop}.

\subsection{The governing equation for \texorpdfstring{\(\eta\)}{}}\label{sec:eta}

According to the definition of $\eta$ in \eqref{eqn:eta}, the primary challenge in characterizing \(\eta_t\) arises from the inverse term \(\left(I - c\Sigma^i_t\right)^{-1}\).
Combining equations~\eqref{eqn:P_to_trace} and~\eqref{eqn:Sigma} gives
\begin{equation}
    (\Sigma^i_t)^2 = \left[\int_t^T\left(\frac{\Tr(\tilde{P}_sL)}{N}\right)^2\,\ud s\right]\Sigma^i_t.
    \label{eqn:Sigma_sq}
\end{equation}
Using assumption~(\hyperref[assu:power_series]{A.4}) and repeatedly applying the above equation produces
\begin{equation}
    \left(I - c\Sigma^i_t\right)^{-1} = I + c\left[1-c\int_t^T\left(\frac{\Tr(\tilde{P}_sL)}{N}\right)^2\,\ud s\right]^{-1}\Sigma^i_t.
    \label{eqn:inv_simplification}
\end{equation}
Plugging it back into equation~\eqref{eqn:eta} and combining it with equation~\eqref{eqn:P_to_trace} yields
\begin{equation}
    \eta_t = \frac{g_t}{1-c\int_t^T g^2_s\,\ud s},
    \label{eqn:eta_in_g}
\end{equation}
where $g_t$ is a scalar function defined as:
\begin{equation}
    g_t := \frac{\Tr(\tilde{P}_tL)}{N}.
    \label{eqn:g}
\end{equation}

A property of $g_t$ will be repeatedly referred to later, so we derive it here. 
Direct calculations based on equation~\eqref{eqn:P_tilde_P} gives:
\begin{equation}
    \Tr(\tilde{P}_tL) = \frac{1}{c\eta_t}\left[\Tr(P_t) - (a+q)\Tr(L)\right].
    \label{eqn:g_rel_eta}
\end{equation}
Given \(\Tr(P_t) = \sum_{k=1}^N \rho^k_t\), combining it with equation~\eqref{eqn:rho} yields:
\begin{equation}
    Ng_t = \Tr(\tilde{P}_tL) 
    = \sum_{k=1}^N \frac{\lambda_k e^{-(T-t)(a+q)\lambda_k}}{1 + c\lambda_k \int_t^T \eta_s e^{-(T-s)(a+q)\lambda_k}\,\ud s}.
    \label{eqn:calc_g}
\end{equation}

Back to equation~\eqref{eqn:eta_in_g}, by multiplying both sides by \(cg_t\), integrating in \(t\), and taking the exponential, we obtain:
\begin{equation}
    e^{c\int_t^T g_s\eta_s\,\ud s} = e^{\int_t^T\frac{cg^2_s}{1-c\int_s^T g^2_u\,\ud u}\,\ud s}= \frac{1}{1-c\int_t^T g^2_s\,\ud s}.
    \label{eqn:exp_int_g_eta}
\end{equation}
Combining equations~\eqref{eqn:g} and~\eqref{eqn:calc_g}, one has
\begin{align}
    c\int_t^T g_s\eta_s\,\ud s
    = \frac{1}{N}\sum_{k=1}^N \int_t^T \frac{c\lambda_k \eta_se^{-(T-s)(a+q)\lambda_k}}{1 + c\lambda_k \int_s^T \eta_u e^{-(T-u)(a+q)\lambda_k}\,\ud u}\,\ud s.
    \label{eqn:int_g_eta}
\end{align}
Combining equations~\eqref{eqn:exp_int_g_eta}--\eqref{eqn:int_g_eta} and using similar arguments as done for equation~\eqref{eqn:calc_rho} implies
\begin{equation}
    \frac{1}{1-c\int_t^T g^2_s\,\ud s}
    =\left[\prod_{k=1}^N \left(1 + c\lambda_k \int_t^T\eta_se^{-(T-s)(a+q)\lambda_k}\,\ud s\right)\right]^{\frac{1}{N}}.
    \label{eqn:equation_g_eta}
\end{equation}
Plugging equations~\eqref{eqn:g},~\eqref{eqn:calc_g} and~\eqref{eqn:equation_g_eta} into equation~\eqref{eqn:eta_in_g} yields
\begin{multline}
    \eta_t = \frac{1}{N}
    \left(\sum_{k=1}^N \frac{\lambda_k e^{-(T-t)(a+q)\lambda_k}}{1 + c\lambda_k \int_t^T \eta_s e^{-(T-s)(a+q)\lambda_k}\,\ud s}\right)\\
    \left[\prod_{k=1}^N \left(1 + c\lambda_k \int_t^T\eta_se^{-(T-s)(a+q)\lambda_k}\,\ud s \right)\right]^{\frac{1}{N}}.
    \label{eqn:eta_sum_prod}
\end{multline}
Finally, rewriting equation~\eqref{eqn:eta_sum_prod} in the matrix form we obtain:
\begin{multline}
    \eta_t = \frac{1}{N}
    \Tr\left[Le^{-(T-t)(a+q)L}\left(I + cL\int_t^T \eta_se^{-(T-s)(a+q)L}\,\ud s\right)^{-1}\right]\\
    \left[\det\left(I + cL \int_t^T\eta_se^{-(T-s)(a+q)L}\,\ud s \right)\right]^{\frac{1}{N}}.
    \label{eqn:eta_charac}
\end{multline}
Observe that equation~\eqref{eqn:eta_charac} involves only the model parameters and the graph Laplacian $L$ without relying on \(P_t\), which aligns with our goal of characterizing $\eta_t$.
Nevertheless, equation~\eqref{eqn:eta_charac} remains complex and does not easily allow for proving existence and uniqueness. To address this, we apply matrix calculus techniques. The lemma below plays an important role in formulating the governing equation for $\eta_t$.

\begin{lem}[{\cite[Equation~(41)]{petersen2008matrix}}]
For a matrix-valued function \(A:[0,T]\to \R^{N\times N}\), if \(A(t)\) takes values as nonsingular matrices for any \(t\in[0,T]\), then
\begin{equation}
    \frac{\ud\det A(t)}{\ud t} = \Tr \left(\mathrm{adj} [A(t)]\frac{\ud A(t)}{\ud t}\right),
\end{equation}
where \(\mathrm{adj} [B] := B^{-1}\det B\) is the adjugate matrix, well-defined for any nonsingular matrix \(B\).

For a matrix-valued function \(A:\R^{N\times N}\to \R^{N\times N}\), if \(A(X)\) takes values as nonsingular matrices for any \(X\in\R^{N\times N}\), then
\begin{equation}
    \frac{\ud\det A(X)}{\ud X} = \frac{\Tr \left(\mathrm{adj} [A(X)] \ud A(X)\right)}{\ud X}.
\end{equation}
\label{lemma:matrix_calc}
\end{lem}

The next corollary follows easily from Lemma~\ref{lemma:matrix_calc} and the chain rule of matrix calculus.

\begin{cor}
    For a matrix-valued function \(R:[0,T]\to \mathbb{S}^{N\times N}\), if \(R(t)\) takes values as 
    nonsingular matrices for any \(t\in[0,T]\), then
    \begin{equation}
    \frac{\ud[\det (I + cLR(t))]^{\frac{1}{N}}}{\ud t}
    = \frac{1}{N}\Tr\left( [I + cLR(t)]^{-1} cLR'(t)\right)[\det (I+cLR(t))]^{\frac{1}{N}}.
    \end{equation}

    For a matrix-valued function \(\X\ni X\mapsto I+cXL\), \(I+cXL\) takes values as nonsingular symmetric matrices for any \(X\in\X\). This leads to
    \begin{equation}
    \frac{\ud[\det (I+cXL)]^{\frac{1}{N}}}{\ud X}
    = \frac{1}{N}[\det (I+cXL)]^{\frac{1}{N}}(I+cXL)^{-1} cL.
    \end{equation}
    
    \label{cor:matrix_calc}
\end{cor}

Recall the definition of \(Q\) from equation~\eqref{eqn:Q}: $Q(X) := \left[\det\left(I+cXL\right)\right]^{\frac{1}{N}}$. If we define the function \(R:[0,T]\to \mathbb{S}^{N\times N}\) as follows:
\begin{equation}
    R(t) := \int_{T-t}^T\eta_se^{-(T-s)(a+q)L}\,\ud s,
\label{eqn:def_R}
\end{equation}
then, based on equation~\eqref{eqn:eta_charac} and Corollary~\ref{cor:matrix_calc} and the definition of $Q$, $R(t)$ solves the following matrix-valued ODE:
\begin{equation}
     R'(t) = \frac{1}{c}\Tr \left[Q'(R(t)) e^{-t(a+q)L}\right]e^{-t(a+q)L},\quad R(0) = 0,  
\end{equation}
which is exactly equation~\eqref{eqn:second_charac}.
Therefore, once a solution $R$ is found, 
\begin{equation}
    R'(T-t) = \eta_{t}e^{-(T-t)(a+q)L}
    \label{eqn:first_R}
\end{equation}
provides the corresponding value of \(\eta_t\).

We conclude this section by summarizing all the assumptions used so far.
\begin{rem}[Assumptions for heuristic derivations]
    \label{rem:assumptions}
    The following assumptions are made in the heuristic construction:

\begin{enumerate}[({A}.1)]
    \item \label{assu:commute}
     \(P_t\) takes values as symmetric matrices and commutes with \(L\) for any \(t\in[0,T]\). 
    \(P_t\) commutes with \(P_s\) for any \(t,s\in[0,T]\).
    
    \item \label{assu:tilde_P_commute}
    \(\tilde{P}_t\) commutes with \(R_\varphi\), \(\forall t\in[0,T]\), \(\forall \varphi\in \text{Aut}(G)\).

    \item \label{assu:invertible}
    \(\Lambda^i_t\) is invertible, \(\forall t\in[0,T]\), \(\forall i\in[N]\).

    \item \label{assu:power_series}
    The power series \(\left(I - c\Sigma^i_t\right)^{-1} = \sum_{n=0}^\infty c^n(\Sigma^i_t)^n\) is valid.
\end{enumerate}
We emphasize again that they do not need to be verified, as the fact that $F_t^i$ given in \eqref{eqn:closed_form_F} solves \eqref{eqn:Riccati} is directly verified in Section~\ref{subsec:verification}. 
\end{rem}

\subsection{The construction of the Markovian NE}\label{sec:close_the_loop}

Once $\eta_t$ has been constructed based on equations~\eqref{eqn:second_charac}--\eqref{eqn:Q}, the Markovian NE can be determined as follows.
From equations~\eqref{eqn:rep_P_use_eta} and~\eqref{eqn:first_R}, \(P_t\) can be represented in terms of \(R\) as follows:
\begin{equation}
    P_t = (a+q)L + R'(T-t)\;cL[I + R(T-t)\;cL]^{-1}.
    \label{eqn:P_first_charac}
\end{equation}
which coincides with equation~\eqref{eqn:R_to_P}.
By combining equations~\eqref{eqn:solution_Riccati_P}--\eqref{eqn:Sigma} and~\eqref{eqn:inv_simplification}, we obtain
\begin{equation}
    F^i_t = \frac{c}{1 - c\int_t^T\left(\frac{\Tr(\tilde{P}(s)L)}{N}\right)^2\,\ud s}\tilde{P}_tLe_ie_i\transpose L\tilde{P}_t.
    \label{eqn:Fi_first_charac}
\end{equation}
The denominator of \(F^i_t\) can be simplified using equations~\eqref{eqn:g}--\eqref{eqn:g_rel_eta}: 
\begin{equation}
    1 - c\int_t^T\left(\frac{\Tr(\tilde{P}_sL)}{N}\right)^2\,\ud s
    = \frac{1}{cN\eta^2_t}[\Tr(P_t) - (a+q)\Tr(L)].
    \label{eqn:denom_fi}
\end{equation}
Plugging equations~\eqref{eqn:P_tilde_P} and~\eqref{eqn:denom_fi} into equation~\eqref{eqn:Fi_first_charac} leads to the result~\eqref{eqn:closed_form_F}.

To this point, we've completed the heuristic derivation of the Markvoian NE for the game \eqref{eqn:state_dynamics}--\eqref{def:J} on vertex-transitive graphs. Specifically, we provide expressions for \(F^i_t\) that depend on \(P_t\), and \(P_t\) that depends on the solution \(R\) to equation~\eqref{eqn:second_charac}.
Next, we aim to justify the well-posedness of equation~\eqref{eqn:second_charac} as well as verifying that equation~\eqref{eqn:closed_form_F} does provide an exact solution to the coupled Riccati system~\eqref{eqn:Riccati}. 

\end{appendices}

\end{document}